\documentclass[letterpaper,10pt]{amsart}

\usepackage[all]{xy}                      %
  
\CompileMatrices                            % Faster

\UseTips                                    % Use

\input xypic
\usepackage[bookmarks=true]{hyperref}       % Hyperref
\usepackage{tikz-cd}
\usepackage{tcolorbox}

\usepackage{amssymb,latexsym,amsmath,amscd}
\usepackage{xspace}
\usepackage{color}
\usepackage{graphicx}
\usepackage{dsfont}
\reversemarginpar

\DeclareMathOperator{\Res}{Res}

\theoremstyle{plain}
\newtheorem{theorem}{Theorem}[section]
\newtheorem*{theorem*}{Theorem}
\newtheorem{proposition}[theorem]{Proposition}

\newtheorem{lemma}[theorem]{Lemma}
\newtheorem{conjecture}[theorem]{Conjecture}

\theoremstyle{definition}
\newtheorem{definition}[theorem]{Definition}

\newtheorem{remark}[theorem]{Remark}
\newtheorem{example}[theorem]{Example}

\newcommand{\enm}[1]{\ensuremath{#1}}          %
\newcommand{\cal}[1]{\mathcal{#1}}

\newcommand{\NN}{\enm{\mathbb{N}}}

\newcommand{\PP}{\enm{\mathbb{P}}}

\newcommand{\TT}{\enm{\mathbb{T}}}
\newcommand{\KK}{\enm{\mathbb{K}}}

\newcommand{\Cc}{\enm{\cal{C}}}

\newcommand{\Ii}{\enm{\cal{I}}}

\newcommand{\Ll}{\enm{\cal{L}}}

\newcommand{\Oo}{\enm{\cal{O}}}

\renewcommand{\phi}{\varphi}
\renewcommand{\theta}{\vartheta}
\renewcommand{\epsilon}{\varepsilon}

\DeclareMathOperator{\red}{red}
\DeclareMathOperator{\reg}{reg}

\renewcommand{\to}[1][]{\xrightarrow{\ #1\ }}

\newcommand{\old}[1]{}

\title{Minimal Terracini loci in projective space}

\author[E. Ballico]{Edoardo Ballico}\address{Dipartimento Di Matematica,
  Universit\`a di Trento, Via Sommarive 14, 38123, Povo, Trento, Italia}
\email{edoardo.ballico@unitn.it}

\author[M.C. Brambilla]{Maria Chiara Brambilla}\address{
Universit\`a Politecnica delle Marche, via Brecce Bianche, I-60131 Ancona, Italia}
\email{brambilla@dipmat.univpm.it}

\thanks{Partially supported by GNSAGA of INdAM}

\begin{document}

\begin{abstract} 
We characterize the number of points for which there exist non-empty Terracini sets of points in $\PP^n$. 
Then we study {\it minimally Terracini} finite sets of points in $\PP^n$ and we obtain a complete description, in the case of $\PP^3$, when the number of points is less than twice the degree of the linear system.
\end{abstract}

\maketitle

%\tableofcontents

\section{Introduction}\label{Sintro}
The notion of {\it Terracini locus} in projective spaces has been recently introduced in 
\cite{BC21} and then
extended to other projective varieties and investigated in 
\cite{bbs,BC2,BV,CG}. 
This property encodes the fact that a set of double points imposes dependent conditions to a linear system, hence it gives information for interpolation problems over double points in special position.

Moreover it can be interpreted in terms of special loci contained in higher secant varieties to projective varieties
as follows.
Recall that
the $k$-th higher secant variety $\sigma_k(X)$ of a
projective variety $X\subset \PP^N$ is the Zariski closure of the
union of all the linear spaces spanned by $k$ independent points
of $X$.
The variety $X$ is called $k$-defective if it has dimension less than the expected one, i.e.\
$\min(N, k \dim(X)+k-1).$
By the famous Terracini lemma \cite{Terracini} a variety is $k$-defective if the tangent spaces to $X$ at $k$ general points 
span a linear space of dimension less than the expected one.
Even when the variety is not $k$-defective, there may be special sets of points such that the span of the tangent spaces drops dimension. We call {\it Terracini} such special sets of points.
For non-defective varieties, we can see the Terracini sets as the points of the abstract secant variety for which
 the
differential of the map to the secant variety is not injective, see e.g.\ \cite{BC21} for more details.

The interest in this subject is also motivated by the connection with the theory of tensors, see e.g.\  \cite{l,BCCGO} for general reference. In particular, since symmetric tensors can be identified with homogeneous polynomials,  the development of geometric methods in projective spaces can give contribution to the study of the rank and decompositions of symmetric tensors.

In this paper we focus on the case of $\PP^n$ and we say that
a finite set of points $S$ of $\PP^n$ is {\it Terracini} with respect to $\Oo_{\PP^n}(d)$ if 
$$h^0(\Ii_{2S}(d)) >0,  \ h^1(\Ii_{2S}(d)) >0,
\mbox{  and }\langle S\rangle =\PP^n.$$
We denote by $\TT(n,d;x)$ the sets of all subsets $S\subset \PP^n$ of cardinality $x$ which are Terracini with respect to $\Oo_{\PP^n}(d)$.

In the language of secant varieties,  the first condition means that the secant variety $\sigma_x(X)\subset \PP^N$ does not fill the ambient space, since $\dim\sigma_x(X)=N-h^0(\mathcal{I}_Z(d))$ (see e.g.\ \cite[Corollary 1]{BCCGO}.
On the other hand, if 
$h^0(\Ii_{2S}(d)) >0$, then
 the number $h^1(\Ii_{2S}(d))$
equals the so-called {\it $x$-defect}, that is  $x(n+1)-\dim\sigma_x(X)-1$ (see Lemma \ref{a7}).

\smallskip

Notice that there are no Terracini sets in $\PP^1$, see Lemma \ref{a2}.
The first result of this paper characterizes the triples $n,d,x$ such that the Terracini locus is non-empty, as follows:

\begin{theorem}\label{ai0}
Fix positive integers $n$, $d$ and $x$.

\quad (i) If either $n=1$ or $d=2$, then $\TT(n,d;x)=\emptyset$ for any $x$.

\quad (ii) $\TT(2,3;x)=\emptyset$ for any $x$.

\quad (iii) 
If $n\ge 2$, $d\ge 3$ and $(n,d)\ne (2,3)$, 
then $\TT(n,d;x)\ne \emptyset$ if and only if $x\ge n+\lceil d/2\rceil$.
\end{theorem}

In order to make a finer description it is very useful to  study {\it minimally Terracini loci}.
The  {\it minimally Terracini} property has been introduced in \cite[Definition 2.2]{bbs} for any projective variety. 
A Terracini set of points $S\subset \PP^n$  is said to be minimally Terracini with respect to $\Oo_{\PP^n}(d)$ if 
$$h^1(\Ii_{2A}(d)) =0\mbox{ for all }A\subsetneq S.$$
We denote by $\TT(n,d;x)'$ the set of all $S\in \TT(n,d;x)$ which are minimally Terracini with respect to $\Oo_{\PP^n}(d)$.

In Theorem \ref{rob1} we see that if $S\in S(\PP^n,x)$ is minimally Terracini for some $\Oo_{\PP^n}(d)$, then such $d$ is unique and 
it is the maximal integer $t$ such that $h^1(\Ii _{2S}(t)) >0$.
%\rosso{(2) [discutere un momento: cancellare o dire meglio?]
%On the other hand, also the dimension $n$ for which a set can be minimally Terracini is unique, by the concision result proved in Proposition \ref{prepa1}.}

Note that, for fixed $n,d$, we know that $\TT(n,d;x)$ is not empty for infinitely many $x$, by Theorem \ref{ai0}. On the other hand, $\TT(n,d;x)'\subseteq \TT(n,d;x)$
is not empty
only for finitely many $x$, as proved in Proposition \ref{oo1}.
In other words the minimality property is a strong condition which 
allows us to prove interesting bounds and characterizations of the triples $n,d,x$ for which $\TT(n,d;x)'$ is or is not empty.

In Section \ref{RNC} we investigate the sets of points on rational normal curves and on their degenerations (reducible rational normal curves). In particular Theorem \ref{a9.0} and Proposition \ref{de2} completely describe the minimal Terracini sets contained in such curves.
Since rational normal curves contain elements of $\TT(n,d;1+\lceil {nd}/{2}\rceil)'$, we may formulate the following conjecture:
\begin{conjecture} For any
$x\le \left\lfloor  \dfrac{nd+1}{2}\right\rfloor$,  
we have  $\TT(n,d;x)'=\emptyset$.
\end{conjecture} 
Here we prove the conjecture for $\PP^2$, Proposition \ref{n2a1}, and for $\PP^3$, Theorem \ref{ooo1}.

After the easy description of the situation in the plane (see Section \ref{sec-plane}), we focus on the case of $\PP^3$, and we obtain the following three results, which are the main results of this paper.
\begin{theorem}\label{ooo1}
Fix integers $d\ge 4$ and $x$ such that $2x\le 3d+1.$ Then $\TT(3,d;x)'=\emptyset$.
\end{theorem}

\begin{theorem}\label{n3.1}
Fix integers $d\ge 7$ and  $x= 1 +\lceil {3d}/{2}\rceil$. Then $S\in \TT(3,d;x)'$ if and only if $S$ is contained in a rational
normal curve.
\end{theorem}

\begin{theorem}\label{ceo1}
Fix integers $d\ge 17$ and $x$ such that $1+\lceil {3} d/2\rceil < x < 2d$. %If $x=2d-1$, assume $d\ge 46$. 
Then $\TT(3,d;x)'=\emptyset$.  %%\blu{sistemare le stime su $d$ PER ORA e' 17} 
\end{theorem}

The bound in Theorem \ref{ceo1} is sharp, as shown in Example \ref{ex4d}, where $2d$ points lie on an elliptic curve.

Summing up,  our results prove that, given $d>0$ and $x\le 2d$, then the minimal Terracini loci $\TT(3,d;x)'$ are empty except for

$\bullet$ either $x= 1 +\lceil {3d}/{2}\rceil$, and in this case the points lie on a rational normal curve,

$\bullet$ or $x=2d$, and in this case the points may lie on an elliptic curve. 

\noindent
We call $(0,1 +\lceil {3d}/{2}\rceil), (1 +\lceil {3d}/{2}\rceil, 2d)$ the first two gaps where the minimal Terracini loci are empty.

 The situation is completely analogous in $\PP^2$, where the first two gaps are $(0,d+1)$ and $(d+1,\lfloor 3d/2\rfloor)$, see Section \ref{sec-plane}.  
 
We expect that a similar behaviour happens also in any dimension $n\ge2$.

\smallskip

The paper is organized as follows:
in Section \ref{prelimi} we present the preliminary results and in particular we introduce the notion of critical scheme, which is a crucial tool in our proofs.
Section \ref{sec 3} contains the first properties of Terracini and minimal Terracini  sets and 
the proof of Theorem  \ref{ai0}.
In Section \ref{RNC} we characterize the minimally Terracini sets of points on rational normal curves and their degenerations. Section \ref{sec-plane} is devoted to the plane, while Section \ref{ultimasez} to the case of $\PP^3$ and to the proofs of Theorems
\ref{ooo1},
\ref{n3.1} and \ref{ceo1}.

\smallskip

We thank the referee for many useful suggestions that improved our presentation.

\section{Preliminaries and notation}\label{prelimi}
We work over an algebraically closed field $\KK$ of characteristic $0$.
For any $x\in\NN$, let $S(\PP^n,x)$ denote the  set of all subsets of cardinality  $x$ of a projective space $\PP^n$.
For any set $E\subset\PP^n$,  let $\langle E\rangle$ denote the linear span of $E$ in $\PP^n$.
\begin{remark}\label{remPE}
It is well-known that the set of configurations of $n+2$ points of $\PP^n$ in linear general position is an open orbit  for the action of $\mathrm{Aut}(\PP^n)$.
\end{remark}

\begin{definition}\label{min1}
We denote by $\TT_1(n,d;x)$ the set of all $S\in S(\PP^n,x)$ such that 
\begin{itemize}
\item
$h^0(\Ii_{2S}(d)) >0$ and $h^1(\Ii_{2S}(d)) >0$.
\end{itemize}
We denote by $\TT(n,d;x)\subseteq \TT_1(n,d;x)$ the set of all $S\in \TT_1(n,d;x)$ such that 
\begin{itemize}
\item  $\langle S\rangle =\PP^n$.
\end{itemize}
We  call {\it Terracini locus} the set $\TT(n,d;x)$ and we say that a finite set $S$  is 
{\it Terracini with respect to} $\Oo_{\PP^n}(d)$ if $S\in\TT(n,d;x)$.
\end{definition}

Obviously  $\TT(n,d;x) =\emptyset$ for all $x\le n$, since every $S\in \TT(n,d;x)$ spans $\PP^n$.

We recall from \cite[Definition 2.2]{bbs} the following important definition; it applies to any projective variety, but we write it now only {in} the case of $\PP^n$.
\begin{definition}\label{min1}
A set $S$ is said to be {\it minimally Terracini with respect to $\Oo_{\PP^n}(d)$} if it is Terracini and moreover
\begin{itemize}
\item $h^1(\Ii_{2A}(d)) =0$ for all $A\subsetneq S$.
\end{itemize}
We denote by $\TT(n,d;x)'$ the set of all $S\in \TT(n,d;x)$ which are minimally Terracini with respect to $\Oo_{\PP^n}(d)$.
\end{definition}

In the next remark we recall the exceptional cases of the Alexander-Hirschowitz theorem, {which are all the cases when any general set of points is minimally Terracini.} 
\begin{remark}\label{a5}  
Assume
$(n,d;x)\in\{(2,4; 5), (3,4;9), (4,4; 14), (4,3; 7)\}$. 
Then by the Alexander-Hirschowitz theorem \cite{ah3} we know that the Veronese variety $\nu_d(\PP^n)$ is $x$-defective.

%%$\dim\sigma_s(X)=N-h^0(\mathcal{I}_Z\otimes\mathcal{O}(d))$

 Fix a general $S\in S(\PP^n,x)$. We have that 
$h^0(\Ii_{2S}(d)) >0$ because the $x$-secant variety does not fill the ambient space,
and $h^1(\Ii_{2S}(d)) >0$ because it is defective. 
Moreover, since $x\ge n+1$, we have $\langle S\rangle=\PP^n$ and hence 
 $S\in \TT(n,d;x)$. 
 
 We prove now that $S$ is minimal.
Indeed, since $S$ is general, then any subset $S'\subset S$ of cardinality $y<x$ is general in
$S(\PP^n,y).$
 Since the secant variety $\sigma _{y}(\nu_d(\PP^n))$ is not defective for any $y\le x-1$, then $h^1(\Ii_{2S'}(d))=0$.
Then we have proved that $S\in \TT(n,d;x)'$.
\end{remark}

We collect here some preliminary results we will use in the sequel.

\begin{lemma}\label{a7}
For any zero-dimensional scheme $Z\subset \PP^n$ and any integer $t\ge0$,
we have $h^i(\Ii_Z(t)) =0$ for all $i\ge2$, and
$$
h^0(\Ii_Z(t)-h^1(\Ii_Z(t))=\binom{n+t}{n}-\deg(Z).
$$
\end{lemma}
\begin{proof}
Since $Z$ is zero-dimensional, we have $h^i(\Oo _Z(t)) =0$ for   all $i>0$. Obviously $h^i(\Oo_{\PP^n}(t)) =0$ for all $i>0$. Then from the exact sequence
$$0\to \Ii_Z(t)\to \Oo_{\PP^n}(t) \to \Oo_Z(t)\to 0$$
we obtain the formulas in the statement.
\end{proof}

\begin{lemma}\label{a9=}
Let $W\subset Z\subset \PP^n$ be zero-dimensional schemes and $t\ge0$. Then we have
$$h^0(\Ii_Z(t))\le h^0(\Ii _W(t))\mbox{\ \ and \ }h^1(\Ii _W(t)) \le h^1(\Ii _Z(t)),$$
and
$$h^0(\Ii_Z(t))\le h^0(\Ii _Z(t+1))\mbox{\ \ and \ }h^1(\Ii _Z(t+1)) \le h^1(\Ii _Z(t)).$$
\end{lemma}

\begin{proof}
Since  $W\subset Z$, then we have the exact sequences
$$0\to \Ii_{W,Z}(t)\to \Oo_W(t) \to \Oo_Z(t)  \to 0,$$
and
$$0\to \Ii_Z(t)\to \Ii_W(t) \to \Ii_{W,Z}(t)\to 0.$$
Since $Z$ is zero-dimensional, then $h^i(\Ii_{W,Z}(d)) =0$ for all $i\ge 1$. Then we get 
$$h^0(\Ii_Z(d))\le h^0(\Ii _W(d))\mbox{\ \ and \ }h^1(\Ii _W(d)) \le h^1(\Ii _Z(d)).$$
From 
the exact sequence
$$0\to \Ii_Z(t)\to \Ii_Z(t+1) \to \Oo_{H}(t+1)\to 0,$$
where $H\subset \PP^n$ is an hyperplane,
it follows that $$h^0(\Ii_Z(d))\le h^0(\Ii _Z(d+1))\mbox{\ \ and \ }h^1(\Ii _Z(d+1)) \le h^1(\Ii _Z(d)).$$
\end{proof}

\begin{lemma}\label{a8}
Given a hyperplane $H\subset \PP^n$ and any finite set $S\subset H$, we have
$$
h^1(\Ii_{2S\cap H,H}(d))\le h^1(\Ii_{2S}(d)) \le h^1(\Ii_{2S\cap H,H}(d))+h^1(\Ii_S(d-1)).
$$
\end{lemma}
\begin{proof}
From
the residual exact sequence with respect to $H$
$$0\to \Ii_S(d-1)\to \Ii_{2S}(d)\to \Ii_{2S\cap H,H}(d)\to 0,
$$
and by Lemma \ref{a7}
the statement
follows.
\end{proof}

{We recall from \cite{bgi} the following useful lemma.}
    \begin{lemma}{\cite[Lemma 34]{bgi}}
    \label{obs1}
Let $Z$ be a zero dimensional scheme in $\PP^n$, such that $h^1(\Ii_Z(d))>0$. 
If $\deg (Z)\le 2d+1$, then there is a line $L$ such that $\deg (Z\cap L)\ge d+2$.
In particular it follows that $\deg(Z)\ge d+2$.
\end{lemma}

We recall the following lemma which we learned from K.\ Chandler (\cite{c0,c}). 

\begin{lemma}\label{chan1}
Let $W$ be an integral projective variety, $\Ll$ a line bundle on $W$ with $h^1(\Ll)=0$ and $S\subset W_{\reg}$ a finite collection of points. Then
$h^1(\Ii _{(2S,W)}\otimes \Ll)>0$ if and only if  there is a scheme $Z\subset 2S$  such that any connected component of $Z$ has degree 
$\le 2$ and such that $h^1(\Ii _Z\otimes \Ll)>0$. 
\end{lemma}

The schemes $Z$ appearing in Lemma \ref{chan1} are curvilinear subscheme of a collection of double points. More precisely in the following definition we introduce the notion of {\it critical schemes}, which are the crucial tools in our proofs.

\begin{definition}\label{def-crit}
Given %$S\in \TT(n,d;x)'$, 
$S$ a collection of $x$ points in $\PP^n$,
we say that a zero-dimensional scheme $Z$ is {\it $d$-critical for $S$} if:
\begin{itemize}
\item $Z\subseteq 2S$ and any connected component of $Z$ has degree  $\le 2$,
\item $h^1(\Ii_Z(d))>0$,
\item $h^1(\Ii_{Z'}(d))=0$ for any  $Z' \subsetneq Z$.
\end{itemize}
\end{definition}

Note that Lemma \ref{chan1}
 implies that for every $S\in \TT(n,d;x)$ there
exists a $d$-critical scheme for $S$.

%\blu{devo separare in due lemmi, uno vale sempre, uno per gli $S$ minimal terracini}

The next lemmas describe the properties of a critical scheme.
\begin{lemma}\label{las0}
Let $Z$ be a zero-dimensional scheme such that
$h^1(\Ii_Z(d))>0$ and $h^1(\Ii_{Z'}(d))=0$ for any  $Z' \subsetneq Z$.
Then $h^1(\Ii_{Z}(d))=1$.
\end{lemma}

\begin{proof}
Assume $h^1(\Ii_Z(d)) \ge 2$ and take a subscheme $Z'\subset Z$ such that $\deg (Z')=\deg(Z)-1$.
We have $h^1(\Ii_{Z'}(d)) \ge h^1(\Ii_Z(d)) -\deg(Z)+\deg(Z')>0$. Thus $Z$ is not critical, a contradiction.
\end{proof}

\begin{lemma}\label{las1}
Fix $S\in \TT(n,d;x)'$ and take $Z$ critical for $S$. Then $Z_{\red}=S$.
\end{lemma}

\begin{proof}
Assume $S':= Z_{\red}\ne S$. Lemma \ref{chan1} gives $h^1 (\Ii_{2S'}(d))>0$. 
Thus $S$ does not belong to $\TT(n,d;,x)'$, a contradiction.
\end{proof}

\begin{lemma}\label{a43} 
Fix integers $n\ge 2$, $d>t\ge 1$ and $x>1$. Take $S\in \TT(n,d;x)'$ and a critical scheme $Z$ for $S$. Take $D\in |\Oo_{\PP^n}(t)|$ with $Z\nsubseteq D$. Then $h^1(\Ii_{\Res_D(Z)}(d-t)) >0$.
\end{lemma}

\begin{proof}
Since $Z\nsubseteq D$ and is critical, then Definition \ref{def-crit} gives $h^1(\Ii_{Z\cap D}(d)) =0$. Thus the residual exact sequence with respect to $D$
gives $h^1(\Ii_{\Res_D(Z)}(d-t)) >0$.
\end{proof}

\section{First results on minimally Terracini sets of points}
\label{sec 3}

We now prove the fact that  if $S\in S(\PP^n,x)$ is minimally Terracini for some $\Oo_{\PP^n}(d)$, then such $d$ is unique and %$d=\tau(2S)$.}
{%%for all $S\in T(n,d;x)'$ the integer $d$ 
it is the maximal integer
$t$ such that $h^1(I_{2S}(t)) >0$.}

\begin{theorem}\label{rob1}
{Fix $n\ge 2$ and $S\in \TT(n,d;x)'$. Then 

\quad (i) $h^1(\Ii _{2S}(d+1)) =0$, 

\quad (ii) $S\notin \TT(n,t;x)$ for any $t\ge d+1$, 

\quad (iii) $S\notin \TT(n,t;x)'$ for any $t\le d-1$.}
\end{theorem}

\begin{proof} %%%[Proof of Theorem \ref{rob1}:]
We now  prove (i) by contradiction.
Assume $h^1(\Ii _{2S}(d+1)) >0$. 
By Lemma \ref{chan1}, there is a $(d+1)$-critical scheme $Z$ for $S$. Recall that, in particular, every component of $Z$ has degree $\le2$. Moreover, by Lemma \ref{las1} we have $S\subset Z\subset 2S$, whereas from Lemma \ref{las0} we know that $h^1(\Ii _Z(d+1)) =1$.

Fix $p\in Z_{\red}$ and call $Z(p)$ the connected component of $Z$ supported at $p$. 
Set $L:= \langle Z(p)\rangle$. Then $L$ is either a line, or a point $L =Z(p) =\{p\}$.

Let $H\subset \PP^n$ be a general hyperplane containing $L$. 
Since $Z$ is curvilinear, by generality of $H$ we can assume that  the scheme $Z\cap H$ is equal to the scheme $Z\cap L$. Let us denote $\zeta=Z\cap H=Z\cap L$.
We will consider separately two possibilities: $h^1(\Ii_{\zeta,H}(d+1)) >0$ and $h^1(\Ii_{\zeta,H}(d+1)) =0$.

\quad {(a)}
Assume first  $h^1(\Ii_{\zeta,H}(d+1)) >0$. 
Then $L$ is a line.   %(giusto perche' se fosse un punto sarebbe non difettivo)
Since $\zeta \subset L$, then we have
the following diagram, whose rows and columns are exact sequences,
\begin{center}
\begin{tikzcd}
\Ii_{L,H}(d+1) \arrow[r] \arrow[d]&\Ii_{\zeta,H}(d+1) \arrow[r] \arrow[d]& \Ii_{\zeta,L}(d+1) \arrow[d]  \\
\Ii_{L,H}(d+1) \arrow[r] &\Oo_H(d+1) \arrow[r] \arrow[d]& \Oo_L(d+1) \arrow[d] \\
  &\Oo_\zeta(d+1) \arrow[r] & \Oo_\zeta(d+1) 
\end{tikzcd}
\end{center}
From the diagram, we get $h^1(\Ii_{\zeta,L}(d+1)) >0$,
which implies $h^1(\Ii_{\zeta}(d+1)) >0$.

Since $\langle S\rangle =\PP^n$ and $n\ge 2$, the set $S'=S\cap L$ is different from $S$.
Now 
by Lemma \ref{chan1} we have that $h^1(\Ii_{2S'}(d+1))>0$, and
Lemma \ref{a9=} implies that $h^1(\Ii_{2S'}(d))>0$.
Hence we have $S\notin \TT(n,d;x)'$, a contradiction. 

\smallskip

\quad {(b)}
Now assume $h^1(\Ii _{\zeta,H}(d+1)) =0$.
In this case the residual exact sequence with respect to $H$ gives
$h^1(\Ii_{\Res_H(Z)}(d)) >0$. 
Since $\Res_H(Z)_{\red} \subseteq S\setminus \{p\}$, by Lemma \ref{chan1} 
we have that $h^1(\Ii _{2(S\setminus\{p\})}(d)) >0$. This contradicts the minimality of $S$, that is 
 $S\notin \TT(n,d,x)'$, a contradiction.

\smallskip

Now it is easy to prove (ii). Indeed by using (i) and  {Lemma} \ref{a9=}, we get, for any $t\ge d+1$, that
$h^1(\Ii_{2S}(t))\le h^1(\Ii_{2S}(d+1))=0$. Hence $S\notin \TT(n,t;x)$.

\smallskip

We prove  (iii) by contradiction. Indeed assume $t\le d-1$ and
$S\in \TT(n,t;x)'$. But then by (i) we have $h^1(\Ii_{2S}(t+1))=0$. Then, since $t+1\le d$ by {Lemma} \ref{a9=} we have
$h^1(\Ii_{2S}(d))\ge h^1(\Ii_{2S}(t+1))=0$, which contradicts the assumption
$S\in \TT(n,d;x)'$.
\end{proof}

The following result is a kind of {\it concision} or {\it autarky} for Terracini loci of Veronese varieties.

\begin{proposition}\label{prepa1}
Take a finite set of points $S\subset \PP^n$ such that $M:= \langle S\rangle \subseteq \PP^n$.  Then
$$h^1(M,\Ii_{2S\cap M,M})(d))> 0 \ \mbox{ if and only if }\ h^1(\Ii_{2S}(d)) >0.$$
\end{proposition}

\begin{proof}
By Lemma \ref{a9=}, we have $h^1(\Ii_{2S}(d)) \ge h^1(\Ii_{2S\cap M}(d))$. Since $M$ is arithmetically Cohen-Macaulay, we get $h^1(M,\Ii_{2S\cap M,M}(d))=h^1(\Ii_{2S\cap M}(d))$. 
Hence the {\it only if} part is obvious.

Now assume $h^1(\Ii_{2S}(d))>0$. Take a hyperplane $H\subset \PP^n$ such that $H\supseteq M$ and use induction on $n-\dim M$.
It is sufficient to prove that $h^1(H,\Ii_{
2S\cap H,H}(d)) >0$.

Take a critical scheme $Z$ for $S$. In order to conclude by Lemma \ref{chan1}, it is enough to find a zero-dimensional scheme $W\subset H$ such that $h^1(H,\Ii_{W,H}(d)) >0$, $W_{\red}=Z_{\red}$ and for each $p\in Z_{\red}$ the connected components, $Z_p$ and $W_p$ of $Z$ and $W$ containing $p$ have the same degree. Fix a general $o\in \PP^n\setminus H$. Let $h_o: \PP^n\setminus \{o\}\to H$ denote the linear projection from $o$.
Since $o$ is general, $o$ is not contained in one of the finitely many lines spanned by the degree $2$ connected components of $Z$. Since $Z_{\red}\subset H$, $o$ is not contained in a line spanned by $2$ points of $Z_{\red}$. Thus ${h_o}_{|Z}$ is an isomorphism. 
Set $W:= h_o(Z)$. By the semicontinuity theorem for cohomology to prove that $h^1(H,\Ii_{W,H}(d)) >0$ it is sufficient to prove that $W$ is a flat limit of a flat family $\{W_c\}_{c\in \KK\setminus \{0\}}$ of schemes projectively equivalent to $Z$. Fix a system $x_0,\dots ,x_n$ of homogeneous coordinates of $\PP^n$ such that $H =\{x_0=0\}$ and $o = [1:0:\dots :0]$. For any $c\in \KK\setminus \{0\}$, let $h_c$ denote the automorphism of $\PP^n$ defined by the formula $h_c([x_0:x_1:\dots :x_n]) =[cx_0:x_1:\dots :x_n]$. Note that $h_{c|H}: H\to H$ is the identity map. Set $W_c:= h_c(W)$.\end{proof}

We start now the classification of Terracini  and minimal Terracini sets of points in $\PP^n$.
Obviously  $\TT(n,d;x) =\emptyset$ for all $x\le n$, since every $S\in \TT(n,d;x)$ spans $\PP^n$.

\begin{lemma}\label{a2} 
$\TT(1,d;x)=\TT_1(1,d;x) =\emptyset$ for all $d>0$ and $x>0$.
\end{lemma}

\begin{proof}
Assume by contradiction the existence of $S\in \TT_1(1,d;x)$. Then $h^1(\Ii_{2S}(d)) >0$ and hence $2x\ge d+2$ and $h^0(\Ii_{2S}(d))>0$ and hence $2x\le d+1$, a contradiction.
\end{proof}

The following proposition shows a key difference between $\TT(n,d;x)$ and its subset $\TT(n,d;x)'$.
In particular for fixed $n$ and $d$, we have
$\TT(n,d;x)'\ne \emptyset$ for only finitely many integers $x$.

\begin{proposition}\label{oo1}
Fix integers $n\ge 2$ and $d\ge 3$. Set $$\rho:= \left\lceil \dfrac{\binom{n+d}{n}+1}{n+1}\right\rceil$$ 
then $\TT(n,d;x)'=\emptyset$ for all $x>\rho$. 
\end{proposition}

\begin{proof}
Let $x>\rho$ and 
assume by contradiction $S\in \TT(n,d;x)'$. 
Then $h^0(\Ii_{2S}(d)) >0$ and, by Lemma \ref{a9=},
$h^0(\Ii_{T}(d)) >0$ for all $T\subseteq S$.
Take $T\subset S$ with $\#(T)=x-1\ge \rho$.
Then we have $h^1(\Ii_{2T}(d)) >0$ by Lemma \ref{a7}.
Then $S$ is not minimally Terracini.
\end{proof}

\begin{lemma}\label{a3} 
$\TT(n,2;x) =\emptyset$ for all $x>0$ and all $n>0$.
\end{lemma}

\begin{proof} 
Assume by contradiction that $S\in \TT(n,2;x)$. Since 
$\langle S\rangle = \PP^n$, we
have $x\ge n+1$.

First assume $x=n+1$. 
Since $\langle S\rangle =\PP^n$, then
the points of $S$ are linearly independent.
Recall that all the quadrics with the same rank are projectively equivalent. 
%Since a general quadric has rank $n+1$, then the Terracini lemma implies $h^0(\Ii_{2S}(2)) =0$.
Since a general quadric form in $\PP^n$ has rank $n+1$, we have that the $(n+1)$-secant variety to $\nu_2(\PP^n)$  fills the ambient space, hence $h^0(\Ii_{2S}(2)) =0$, and this contradicts the fact that $S$ is Terracini.

Now assume $x\ge n+2$. Since $\langle S\rangle =\PP^n$, there exists a subset $S'\subset S$ of cardinality $n+1$ and such that $\langle S'\rangle=\PP^n$.
We just proved that $h^0(\Ii_{2S'}(2)) =0$. 
By Lemma \ref{a9=} we deduce that $h^0(\Ii_{2S}(2)) =0$.
\end{proof}

The following result shows that {many elements of $\TT_1(n,d;x) \setminus \TT(n,d;x)$ are easily produced} and not interesting.
\begin{lemma}\label{a4} 
Fix $n\ge 2$, $d\ge 2$ and $x\ge \lceil d/2\rceil +1$. {Let $S$ be a collection of $x$ points on a line $L\subset \PP^n$}. Then $S\in \TT_1(n,d;x)$.
\end{lemma}

\begin{proof}
We need to prove that $h^1(\Ii_{2S}(d)) >0$ and $h^0(\Ii_{2S}(d)) >0$. Fix a hyperplane $H$ containing $L$. Take $G:= 2H$ if $d=2$ and call $G$ the union of $2H$ and a hypersurface of degree $d-2$ if $d>2$.
Since $S\subset \mathrm{Sing}(G)$, we have $h^0(\Ii_{2S}(d)) >0$. Since $\deg(2S\cap L)  =2x\ge d+2$, $h^1(\Ii _{2S\cap L}(d)) >0$. Thus $h^1(\Ii_{2S}(d)) >0$, by Lemma \ref{a9=}.
\end{proof}

\begin{lemma}\label{ooo3}
For any $x>0$, we have $\TT(3,3;x)'=\emptyset$. 
\end{lemma}
\begin{proof}
The case  $x\le4$ will be treated 
in Proposition \ref{43}.

Fix now $x\ge5$ and assume by contradiction that there exists $S\in \TT(3,3;x)'$. If four of the points of $S$ are in a plane $H$, then  $S$ is not minimal. Indeed if $A$ is the union of the four points in the plane, then  $h^1(\Ii_{2A}(3))=
h^1(\Ii_{2A\cap H,H}(3))>\deg(2A\cap H)-\binom{2+3}{2}=12-10=2$ by Proposition \ref{prepa1} and Lemma \ref{a7}.

Therefore the points of $S$ are in linearly general position. Consider $S'\subseteq S$ of cardinality $5$. The points of $S'$ are in linearly general position and, by Remark \ref{remPE}, they are projectively equivalent to a general set of five points $A$ of $\PP^3$.

Since the Veronese variety $\nu_3(\PP^3)$ is not defective, by Alexander-Hirschowitz theorem, we know that
$\sigma_4(\nu_3(\PP^3)$ fills the ambient space. Hence $h^0(\Ii_{2A}(3)) =0$. Then $h^0(\Ii_{2S'}(3)) =0$  and by Lemma \ref{a9=} we get
$h^0(\Ii_{2S}(3)) =0$, and this contradicts the fact that $S$ is Terracini.
\end{proof}

\subsection{Proof of Theorem  \ref{ai0}}
We are now in position to give the proof of Theorem \ref{ai0} which classifies  Terracini loci. We start with the following  lemma.

\begin{lemma}\label{43-}
Assume $n\ge1$ and $d\ge 2$. Let $Z\subset \PP^n$ be a zero-dimensional scheme such that $\deg(Z) \le d+n+1$, $h^1(\Ii_Z(d)) >0$  and  $\langle Z\rangle = \PP^n$. Then there is a line $L$ such that $\deg (L\cap Z)\ge d+2$
and $\deg (Z)= d+n+1$.
\end{lemma}

\begin{proof}
The lemma is trivial for $n=1$.

{We prove the statement by induction on $n\ge 2$.}
First we assume $n=2$. Since $\deg (Z)\le 2d+1$, there is a line $L$ such that $\deg (Z\cap L)\ge d+2$, by Lemma \ref{obs1}. Clearly, since $\langle Z\rangle =\PP^2$, we get $\deg (Z) =d+3$.

Now assume  $n>2$. Take a hyperplane $H\subset \PP^n$ such that $w:= \deg(Z\cap H)$ is maximal. Since  $\langle Z\rangle = \PP^n$, we have  $n\le w<z$ and $\langle Z\cap H\rangle=H$. 

If $h^1(\Ii_{Z\cap H,H}(d)) >0$, then by induction we have that
 there is a line $L$ such that $\deg (L\cap (Z\cap H))\ge d+2$
and $\deg (Z\cap H)= d+n$. Hence it follows that
$\deg (L\cap Z)\ge d+2$
and $\deg(Z)\ge d+n+1$, hence $\deg(Z)= d+n+1$.

Now assume $h^1(\Ii_{Z\cap H,H}(d)) =0$ and by the residual exact sequence with respect to $H$
\begin{equation}\label{eq==2}
0\to \Ii_{\Res_H(Z)}(d-1)\to \Ii_Z(d)\to \Ii_{Z\cap H,H}(d)\to 0,
\end{equation}
we have $h^1(\Ii_{\Res_H(Z)}(d-1))>0$.
By Lemma \ref{obs1}, since $\deg (\Res_H(Z))\le z-w\le d+1\le 2d+1$, we have a line $L$ with $\deg (L\cap \Res_H(Z)) \ge d+2$. Since $\langle Z\rangle=\PP^n$, we must have $\deg(Z)\ge \deg(Z\cap L)+n-1\ge d+n+1$. Hence the assumption
$\deg(Z)\le d+n+1$ implies that
$\deg (Z)= d+n+1$ and
$\deg (L\cap Z)\ge d+2$.
\end{proof}
%\blu{in pratica il lemma 34 dice che nessuno schema di grado minore di $d+1$ puo' avere $h^1$}
The following Proposition \ref{43} proves the emptyness of the Terracini locus for small number of points.
We give first a numerical lemma which will be used in the proof of the proposition.

%\rosso{
%\begin{lemma}\label{lem-num}
%Given  $x,y,m,n,d\in \NN$, such that
%$d\ge3$, $n\ge3$, $y\le x$, $m<n$, $y\le x+m-n$,
%$x \le n+\lceil \frac{d}{2}\rceil-1$,  
%$y(m+1)\ge \binom{m+d}{m}$, 
% then we have $m=1$.
%\end{lemma}
%\begin{proof}
%Using the assumption we have
%$$\binom{m+d}{m}\le y(m+1)\le \left(m+\left\lceil \frac{d}{2}\right\rceil -1\right)(m+1)
%\le \left(m+\frac{d}2\right)(m+1)$$
%% $(m+1)(m+\left\lceil d/2\right\rceil -1) \ge \binom{m+d}{m}$
%e sai anche $d\ge 3$
%%%$(m+1)(m+d/2)\ge \binom{m+d}{m}$ 
%per $m$ fissato facendo le differenze prime basta $d=3$ che e' ok se $m=3$
%\\
%ora fisso $m=3$ e cerco assurdo
%\\
%We prove now by induction on $m\ge3$ that 
%\begin{equation}\label{formuletta}\binom{m+d}{m}>\left(m+\frac{d}2\right)(m+1)
%\end{equation}
%It is easy to check that \eqref{formuletta} is true for $m=3$ and any $d\ge3$.
%Now we assume \eqref{formuletta} for $m$ and we have, by using the induction hypothesis:
%$$
%\binom{m+1+d}{m+1}=\binom{m+d}{m+1}+\binom{m+d}{m}=\binom{m+d}{m}\left(\frac{d}{m+1}+1\right)
%>$$
%$$>\left(m+\frac{d}2\right)(m+1)\left(\frac{d}{m+1}+1\right)=
%\left(m+\frac{d}2\right)(d+m+1)=$$
%$$=\left(m+\frac{d}2\right)(m+2)+\left(m+\frac{d}2\right)(d-1)\ge
%\left(m+\frac{d}2+1\right)(m+2)$$
%where the last inequality holds for $d\ge3$.
%\\
%in realta' per il lemma 2.7 si sa anche che $(m+1)y > \binom{m+d}{m}$
%quindi se $y\le n-m+x = m+\lceil d/2\rceil -1$
%per $d=3$ si ha $(m+1)(m+1) >\binom{m+3}{m}$
%che e' falso anche per $m=2$
%\end{proof}
%}

\begin{lemma}\label{lem-num}
Given  $x,y,m,n,d\in \NN$, such that
$d\ge3$,  $n\ge3$, $m<n$, $x\ge y+n-m$, 
$x \le n+\lceil \frac{d}{2}\rceil-1$,  
$y(m+1)\ge \binom{m+d}{m}$, 
 then we have $m=1$.
\end{lemma}

\begin{proof}
Using the assumptions,  in particular $y\le x-(n-m)\le x-1$, we have
$$\binom{m+d}{m}\le (x-1)(m+1)\le\left(m+\left\lceil \frac{d}{2}\right\rceil -2\right)(m+1)
\le \left(m-1+\frac{d}2\right)(m+1).$$
We prove now by induction on $m\ge2$ that 
\begin{equation}\label{formuletta}\binom{m+d}{m}>\left(m-1+\frac{d}2\right)(m+1).
\end{equation}
It is easy to check that \eqref{formuletta} is true for $m=2$ and any $d\ge3$.
Now we assume \eqref{formuletta} for $m$ and we have, by using the induction hypothesis:
$$
\binom{m+1+d}{m+1}=\binom{m+d}{m+1}+\binom{m+d}{m}=\binom{m+d}{m}\left(\frac{d}{m+1}+1\right)
>$$
$$>\left(m-1+\frac{d}2\right)(m+1)\left(\frac{d}{m+1}+1\right)=
\left(m-1+\frac{d}2\right)(d+m+1)=$$
$$=\left(m-1+\frac{d}2\right)(m+2)+\left(m-1+\frac{d}2\right)(d-1)\ge
\left(m+\frac{d}2\right)(m+2)$$
where the last inequality holds because
$\left(m-1+\frac{d}2\right)(d-1)\ge(m+2)$
for any $d\ge3$ and $m\ge1$.
\\
Hence since we have proved  \eqref{formuletta} for any $m\ge2$, we conclude that $m=1$.
\end{proof}

\begin{proposition}\label{43}
Assume $n,d\ge2$ and fix an integer $x$ such that $$x \le n+\left\lceil \frac{d}{2}\right\rceil-1.$$ Then $\TT(n,d;x)=\emptyset$.
\end{proposition}

\begin{proof} 
The case $d=2$ is true by Lemma \ref{a3}, hence we can assume $d\ge3$.

Assume $n=2$. Assume by contradiction that
$S\in \TT(2,d;x)$. Let $Z$ be a critical scheme for $S$, then we have $\deg(Z)\le 2x\le d+3$. Hence by Lemma \ref{43-} there exists a line $L$ such that
 $\deg (Z\cap L) \ge d+2$ and hence $x> \#(S\cap L)\ge \lceil {d}/{2}\rceil+1$, a contradiction.

Assume $n\ge 3$ and use now induction on $n$. 
{By contradiction assume}
$S\in \TT(n,d;x)$. Let $S'\subseteq S$ be the minimal subset such that $h^1(\Ii_{2S'}(d)) >0$. Set $y:= \#S'$, $M:= \langle S'\rangle$,
and $m:= \dim M$. Proposition \ref{prepa1} gives $h^1(\Ii_{2S'\cap M,M}(d)) >0$. 
Notice that Lemma \ref{chan1} and Lemma \ref{obs1} imply that $2y\ge d+2$.

\smallskip

(I)
If $m<n$, then we consider two cases.

\smallskip

\quad(a)
If $h^0(M,\Ii_{2S'\cap M}(d)) >0$, then we have $y\ge m+\lceil d/2\rceil$ by the induction assumption. Then $x\ge y+(n-m)=n+\lceil d/2\rceil$, a contradiction.

\smallskip

\quad(b) 
If
 $h^0(M,\Ii_{2S'\cap M}(d)) =0$, then $y(m+1)\ge \binom{m+d}{m}$.
 Since $S$ spans $\PP^n$, then $x\ge y+(n-m)$. Hence 
  by Lemma \ref{lem-num}, we get $m=1$. Then $M$ is a line 
and in this case we have again a contradiction because, since $2y\ge d+2$, we have
 $$x\ge y+(n-1) \ge \frac{d+2}2+n-1=n+\frac{d}2.$$

\smallskip

(II)
Thus we may assume $m=n$.  
Let $H\subset \PP^n$ be any hyperplane such that $H$ is spanned by $S'\cap H$. {Let $S''= S'\cap H$. {Then} $n\le \#(S'')< y$.}  
Since $\Res_H(2S'') =S''$, we have the exact sequence:
\begin{equation}\label{succesatta}0\to \Ii_{S''}(d-1)\to\Ii_{2S''}(d)\to \Ii_{2S''\cap H,H}(d)\to0.\end{equation}
The minimality of $S'$  and Proposition \ref{prepa1} give $h^1(H,\Ii_{{2S''\cap H},H}(d)) =0$. 

\smallskip

{(a)} Now,
 if $h^1(\Ii_{S''}(d-1)) >0$, then

{\quad(a.1)} either
$\#(S'') \ge n+d$, {which gives} a contradiction {with $x\le n+\lceil \frac{d}{2}\rceil-1$};

{\quad(a.2)}
or $\#(S'') \le n+d-1$. 
In the latter case, Lemma \ref{43-} applied {to $S''\subset H$} gives $\#(S'') = n+d-1$, which  also contradicts $x\le n+\lceil {d}/{2}\rceil-1$.

\smallskip

{(b)}
Hence  we may assume $h^1(\Ii_{S''}(d-1)) =0$. From the exact sequence \eqref{succesatta}, we get
$h^1(\Ii _{2S''}(d)) =0$.

We consider now the {residual} exact sequence with respect to the quadric hypersurface $2H$:
$$0\to \Ii_{S'\setminus S''}(d-2)\to\Ii_{2S'}(d)\to \Ii_{2S'',2H}(d)\to0$$
where
$\Res_{2H}({2}S') =
S'\setminus S''$. 

 Since the quadric hypersurface $2H$ in $\PP^n$ is arithmetically Cohen-Macaulay, we get $h^1(\Ii_{2S'',2H}(d)) =0$, which implies %$h^1(\Ii_{\Res_{2H}(Z)}(d-2))>0$. 
 $h^1(\Ii_{S'\setminus S''}(d-2))>0$. Then by Lemma \ref{obs1}, we have 
 $\deg(S'\setminus S'')\ge d$.
 But since $\deg(S'\setminus S'')=\#(S'\setminus S'')\le y-n\le \lceil d/2\rceil-1$, we have a contradiction, since $\lceil d/2\rceil-1< d$ for all $d\ge2$. 
\end{proof}

We now give the proof of the main result of this section.
 
\begin{proof}[Proof of Theorem \ref{ai0}:]
Part (i) is true by Lemmas \ref{a2} and \ref{a3}.

We prove now part (ii).  Assume $n=2$ and $d=3$.
A singular plane cubic $C$ {with at least $3$ singular points} is either the union of $3$ lines or a triple line or the union of a double line and another line. Thus if $\mathrm{Sing}(C)$ spans $\PP^2$, then $\#\mathrm{Sing}(C) =3$  and $\mathrm{Sing}(C)$ is projectively equivalent to {any configuration of} $3$ non-collinear points. Hence $\TT(2,3;x)=\emptyset$ for all $x\ge 4$. % \rosso{and we have proved (iii).}
Thus we have proved (ii) because clearly $\TT(2,3;3)=\emptyset$.

For part (iii), assume that $n\ge 2$, $d\ge 3$ and $(n,d)\ne(2,3)$. 
By Proposition \ref{43} we have that if $x< n+ \lceil d/2\rceil$, then $\TT(n,d;x)= \emptyset$. 
Hence it is enough to prove {that
 $\TT(n,d;x)\ne \emptyset$ for $x\ge n+\lceil d/2\rceil$.}

{We now analyse three different cases separately.} 

\smallskip

\quad (I)
Consider first the case $n=2$ and $d\ge4$. 
We {assume} $x\ge  \lceil d/2\rceil+2$.
Let $L,M,N$ {be three distinct lines} and $G:= (d-2)L\cup M\cup N$.
Take as $S$ {the union of} the point $M\cap N$ and $x-1$ points on $L\setminus (M\cup N)$. 
Since $S\subset \mathrm{Sing}(G)$, then $h^0(\Ii_{2S}(d)) >0$. 
{Furthermore we claim that}  $h^1(\Ii_{2S}(d)) >0$. Indeed $L$ contains {at least }$\lceil d/2\rceil +1$ points of $L$, hence $\deg(2S\cap L)\ge d+2$ and by {Lemma} \ref{a8} we have
$h^1(\Ii_{2S}(d))\ge h^1(\Ii_{2S\cap L,L}(d))>0.$
{Summing up, since $\langle S\rangle=\PP^2$ 
we get $S\in \TT(2,d;x)$, i.e.\
$\TT(2,d;x)\ne \emptyset$.}

\smallskip

\quad (II)
Now assume $n\ge 3$, $d=3$ and $x\ge n+2$. 
Fix hyperplanes $H, K, U$ of $\PP^n$ such that $\dim H\cap K\cap U =n-3$. 
Since $H\cap K$ and $H\cap U$ are $2$ different codimension $1$ subspaces of $H$, their union spans $H$.

Let $S$ be the union of 
{$n$ general points in $(H\cap K)$, one point in $(H\cap U)\setminus(H\cap K\cap U)$
and a point in $(K\cap U)\setminus (H\cap K\cap U)$}.
{Then} $\langle S\rangle=\PP^n$, $h^0(\Ii_{2S}(3))\neq0$ and it is easy to show (by induction on $n$) that $h^1(\Ii_{2S}(3))\neq0$. %\blu{ti torna? aggiungere piu' dettagli?}
Hence by {Lemma} \ref{a9=}, 
{for any configuration $S'$ of points such that $S\subset S'\subset \text{Sing}(L\cup M\cup N)$ we have $S'\in \TT(n,d;\#(S'))$. In consequence,}
 $\TT(n,3;x)\ne \emptyset$ for all $x\ge n+ 2$ and $n\ge3$.

\smallskip

\quad (III)
Now assume $n\ge3$, $d\ge4$ {and} $x\ge n+ \lceil d/2\rceil$.
{As before, fix hyperplanes $H,K,U$ with $\dim(H\cap K\cap U)=n-3$ and}
take a line $L\subset H$ and set $G:= (d-2)H\cup K\cup U$. 
Consider a collection $E$ of $x-n+1$ points on the line $L$. 
Since $\#E \ge \lceil d/2\rceil+1$, by Lemma \ref{a4} we have $h^1(\Ii_{2E}(d)) >0$. Let $A\subset H$ be a collection of $n-2$ general %\blu{o preferisci linearly independent?} 
points. Note that $\langle E\cup A\rangle =H$. Take as $S$ the union of $A\cup E$ and a point of $(U\cap K)\setminus (H\cap K\cap U)$. Obviously $S$ spans $\PP^n$ and $h^1(\Ii _{2S}(d)) >0$ by {Lemma} \ref{a9=}. 
{Moreover $h^0(\Ii _{2S}(d)) >0$  by construction and in consequence $S\in\TT(n,d;x)\ne\emptyset$.}
\end{proof}

Notice that the set of points $S\in\TT(3,3;5)$ produced in the previous proof is not minimally Terracini, because $4$ points belong to a plane. Indeed by Lemma \ref{ooo3} we already know that $\TT(3,3;5)'=\emptyset$.

\section{Rational normal curves}\label{RNC}
We start now to analyze the set of points lying on a rational normal curve.
For each $n>1$, we denote by $\Cc_n$ the set of all rational normal curves of $\PP^n$.

\begin{lemma}\label{a9.00}
Fix integers $n\ge 2$, $d\ge 4$ and $x\le \lceil {nd}/{2}\rceil$. Take a rational normal curve $C\in \Cc_n$ and let $S\subset C$ be a collection of $x$ points on $C$. Then $h^1(\Ii_{2S}(d)) =0$.
\end{lemma}

\begin{proof}
Assume {by contradiction that} $h^1(\Ii_{2S}(d)) >0$. By Lemma \ref{chan1}, 
there exists a $d$-critical scheme $Z$ for $S$.
Since $C$ is scheme-theoretically cut-out by quadrics, there is $Q\in |\Ii_C(2)|$ such that $Q\cap Z =C\cap Z:=\zeta$ and we have
$$0 \to \Ii_{C,Q}(d) \to \Ii_{\zeta,Q}(d) \to \Ii _{\zeta,C}(d)\to 0.$$
Since 
$\deg (Z) \le 2x\le nd+1$, we have
$h^1(\Ii_{\zeta,C}(d)) =0$, and since $C$ is projectively normal we get $h^1(\Ii_{\zeta,Q}(d)) =0$. 
Thus the residual exact sequence with respect to $Q$ and 
{the fact that $h^1(\Ii_Z(d))>0$}
give $h^1(\Ii_{\Res_Q(Z)}(d-2)) >0$.

Since $\Res_Q(Z)\subseteq S\subset C$, we have
$$0 \to \Ii_{C}(d-2) \to \Ii_{\Res_Q(Z)}(d-2) \to \Ii _{\Res_Q(Z),C}(d-2)\to 0.$$
{We have $h^1(\Ii_{C}(d-2))=0$, because $C$ is projectively normal.}

{Note that}
$\deg (\Res_Q(Z))< n(d-2)+2$, indeed 
$$\deg (\Res_Q(Z))\le x\le \left\lceil \frac{nd}{2}\right\rceil\le n(d-2)+1,$$
where the last inequality is true for $d\ge 4$.
{Then we have
$$h^1(\Ii_{\Res_Q(Z),C}(d-2))=h^1(\Oo_{\PP^1}(n(d-2)-\deg(\Res_Q(Z))))=0,$$ and we have
a contradiction with $h^1(\Ii_{\Res_Q(Z)}(d-2)) >0$.}
\end{proof}

\begin{theorem}\label{a9.0} 
Fix integers $n\ge 2$, $d\ge 3$ and assume $(n,d)\ne (2,3)$.
Given a rational normal curve $C\in \Cc_n$ and a collection  $S\subset C$ of $x$ points on the curve.
Then 

\quad (i)  if $n\ge3, d\ge4$  and $x\ge1+\lceil {nd}/{2}\rceil$, then {$S\in \TT(n,d;x)$};

\quad (ii) if $n\ge4, d=3$
and $x=1+\lceil {nd}/{2}\rceil$, then $S\in \TT(n,d;x)$;

\quad (iii) if $n\ge2$, $d\ge 4$  and $x=1+\lceil {nd}/{2}\rceil$, then $S\in \TT(n,d;x)'$.
\end{theorem}

\begin{proof}
By the exact sequence
$$0 \to \Ii_{C\cup 2S}(d) \to \Ii_{2S}(d) \to \Ii _{2S\cap C,C}(d)\to 0$$
since  
$h^1(\Ii_{2S\cap C,C}(d))=h^1(\Oo_{\PP^1}(nd-2x))>0$
{whenever $x\ge1+\lceil {nd}/{2}\rceil$,}
we have
$h^1(\Ii_{2S}(d)) >0$.
Since $x\ge n+1$ and $C$ is a rational normal curve, then $\langle S\rangle =\PP^n$. 

If $n\ge3$ then $h^0(\Ii_C(2))\ge2$, hence $C$ is contained {in a quadric hypersurface}.
Thus if $d\ge 4$, we have $h^0(\Ii _{2S}(d)) >0$.  
Hence $S\in \TT(n,d;x)$ and we have proved (i).

Assume now $x=1+\lceil {nd}/{2}\rceil$. Fix a collection $A$ of $x$ general points {on $C$} and 
note that {by generality}
$h^0(\Ii _{2S}(d))\ge h^0(\Ii _{2A}(d))$.

Hence, assuming  $d=3$ we have
$$h^0(\Ii _{2S}(3))\ge \binom{n+3}3-(n+1)x> 0$$
where the last inequality is true for any
$n\ge5$. 
If $n=4$ and $x=7$, we have $h^0(\Ii _{2S}(3))\ge h^0(\Ii _{2A}(3))=1$, by the Alexander-Hirschowitz theorem. {In consequence $S\in\TT(n,3:x)$, which ends the proof of}  (ii).

Now assume $n=2$ and $x= d+1$. 
We have $h^0(\Ii _{2S}(d))\ge \binom{d+2}2-3 (d+1)>0$, for $d\ge5$.
If $d=4$ and $x=5$, then
$h^0(\Ii _{2S}(4))\ge h^0(\Ii _{2A}(4))=1$, again by the Alexander-Hirschowitz theorem.
Hence $S \in \TT(n,d;x)$ for $n=2$ and $d\ge4$.

In order to complete the proof of (iii) we need to prove the minimality of $S$, and this follows by Lemma \ref{a9.00}.
\end{proof}

\begin{remark}
Recall that by Theorem \ref{ai0} we know that $\TT(2,3;x)=\emptyset$ for all $x>0$.
Moreover in the proof of Lemma \ref{ooo3}, we have seen that {a set of}
$x\ge 5$ {points  in linearly general  position in} $\PP^3$ is not Terracini. Hence if $S$ is a collection of $x\ge 5$ points on a rational normal cubic {curve} we have $S\not\in\TT(3,3;5)$.
\end{remark}

%
%The following show that there are examples of $S\in \TT(n,d;x)'$ with different critical schemes:
%\begin{example}\label{rtn0e1} \rosso{to do. da tenere?}
%Fix integers $n\ge 3$ and $d\ge 3$ with $n$ and $d$ odd. Set $x:= 1 +(nd+1)/2$. Take $C\in \Cc_n$ and any $S\subset C$ such that $\#S = x$. Fix $o\in S$ and set $S':= S\setminus \{o\}$. As in Example \ref{a9} we see that $\{o\}\cup (2S',C)$ is a critical scheme for $z$.
%\end{example}

\subsection{Degenerations of rational normal curves.} 
We introduce now the notion of {\it reducible rational normal curves}. 

\begin{definition}
 A reduced, connected  and reducible curve $T\subset \PP^n$, for $n\ge 2$, such that $\deg (T) =n$ {and}  $\langle T\rangle =\PP^n$ is called {\it reducible rational normal curve}.
 \end{definition}
 
%\blu{ ma vanno bene tutte? per esempio tre rette per lo stesso punto? non dobbiamo dire "rational"?(11)}
 
 Of course, in $\PP^2$ a reducible rational normal curve is a reducible conic. 
 
Since $T$ is connected, there is an ordering $T_1,\dots ,T_s$ of the irreducible component such that each $T[i]:= T_1\cup \cdots \cup T_i$, $1\le i \le s$, is connected.
We say that each such ordering of the irreducible components of $T$ is a {\it good ordering}. 

Set $n_i:= \deg (T_i)$. Note that $n =n_1+\cdots +n_s$ and $\dim \langle T_i\rangle \le n_i$ with equality if and only if $T_i$ is a rational normal curve in its linear span. 
%Hence if $\dim \langle T_i\rangle =n_i$, then $T_i$ is smooth and a rational normal curve in its linear span. 
For $i=1,\dots ,s-1$ we have the following Mayer-Vietoris exact sequence
\begin{equation}\label{eqde1}
0 \to \Oo_{T[i+1]}(t) \to \Oo_{T[i]}(t)\oplus \Oo_{T_{i+1}}(t) \to \Oo _{T[i] \cap T_{i+1}}(t)\to 0,
\end{equation}
in which $T[i]\cap T_{i+1}$ is the scheme-theoretic intersection. Since $T[i+1]$ is connected, $\deg (T[i]\cap T[i+1])>0$. Thus \eqref{eqde1} gives $\dim \langle T[i+1]\rangle \le \dim \langle T[i]\rangle +n_i$ with equality if and only if
$\deg (T[i]\cap T[i+1])=1$, $T_{i+1}$ is a rational normal curve in its linear span and  $\langle T[i]\rangle \cap \langle T_{i+1}\rangle$ is the point $T[i]\cap T_{i+1}$. 

Since $n =n_1+\cdots +n_s$, by induction on $i$ we get $p_a(T)=0$ and each $T_i$ is a rational normal curve in its linear span. Using \eqref{eqde1} and induction on $t$ we also get $h^1(\Oo_T(t)) =0$ and $h^0(\Oo_T(t)) =nt+1$ for all $t\ge 0$,  and that the restriction map $H^0(\Oo_{\PP^n}(t)) \to H^0(\Oo _T(t))$ is surjective, i.e. $T$ is arithmetically Cohen-Macaulay. In the same way we see that each $T[i]$ is arithmetically Cohen-Macaulay in its linear span. 

%Fix an integer $d\ge 2$. Since $h^0(\Oo_T(t)) =nt+1$ for all $t\ge 0$, $h^1(T,\Ii_{Z,T}(d)) >1$ if $\deg (Z)\ge nd+2$. 

Recall that each $T_i$ is smooth. For any $p\in T_i$ let $L_i(p)$ denote the tangent line of $T_i$ at $(p)$. Take $p\in \mathrm{Sing}(T)$ and let $T_{i_1}, \dots T_{i_k}$, $k\ge 2$, be the irreducible components of $T$ passing through $p$. Since $n=n_1+\cdots +n_s$ and $p_a(T)=0$, the $k$ lines $L_{i_1}(p),\dots ,L_{i_k}(p)$ through $p$ span a $k$-dimensional linear space (such a singularity is often called a seminormal or a weakly normal curve singularity). 

An irreducible component $T_i$ of $T$
is said to be a {\it final component} if $\#(T_i\cap \mathrm{Sing}(T)) =1$. Since $s\ge 2$, $T$ has at least $2$ final components (e.g. $T_1$ and $T_s$ for any good ordering of the irreducible components of $T$), but it may have many final components
(e.g. for some $T$ with $s\ge 3$ we may have {$\#(T_i\cap\mathrm{Sing}(T)) =1$} for all $i\ge 2$ and there is one $T$, unique up to a projective transformation, formed by $n$ lines through the same point). 

%\blu{mi sembra che questo paragrafo si possa togliere, non usiamo mai la chain}
%We say that $T$ is {\it a chain} if it has exactly $2$ final components. If $T$ is a chain, then it has exactly $2$ good ordering and if $T_1,T_2,\dots ,T_s$ is one good ordering, then the other one is $T_s,T_{s-1},\dots ,T_s$. Since a good ordering of a reducible rational normal curve $T$ may start from any final component, $T$ is a chain if and only if it has exactly $2$ good orderings. 
%For any $s\ge 2$ and any positive integers $n_1,\dots ,n_s$
%there is a chain $T$ with degrees $n_1,\dots ,n_s$ of its irreducible components and any two such chain reducible rational normal curves (with a fixed good ordering) are projectively equivalent. 

\begin{remark}\label{de001}
Take a (reducible) rational normal curve $T\subset \PP^n$. Since $h^1(\Oo_T)=0$, the exact sequence $$0\to \Ii_T \to \Oo_{\PP^n} \to \Oo_T\to 0$$ gives $h^2(\Ii _T)=0$. Since $h^1(\Ii _T(1))=0$, the Castelnuovo-Mumford Lemma {implies} that the homogeneous ideal of $T$ is generated by quadrics. Thus $T$ is scheme-theoretically cut out by quadrics.
\end{remark}

\begin{lemma}\label{de1} 
Fix $n\ge 2$, $d\ge4$.
% and $(n,d)\neq(2,3)$. \rosso{dove sappiamo che e' vuoto}
Let $T$ be a reducible rational normal curve in $\PP^n$ and $S\in S(\PP^n,x)$ such that $S\subset T_{\reg}$ and $\langle S\rangle =\PP^n$. If $2x\ge dn+2$, then $S\in \TT(n,d;x)$.
\end{lemma}

\begin{proof}
% \blu{aggiungo che $h^0>0$}
Since $h^0(\Ii_T(2))=\binom{n}2$, we have that $h^0(\Ii_{2S}(d))>0$ if $d\ge4$. 
%\rosso{prendo due quadriche} 
%
%\blu{pero' fare a parte $\PP^2$ any $d\ge4$ e $d=3$ any $n\ge 2$... da fare (2)}\rosso{TO DO}

Set $Z:= 2S\cap T$. Since $S\cap \mathrm{Sing}(T)=\emptyset$, $\deg (Z) =2x$ and $Z$ is a Cartier divisor of $T$. 
 Since $h^0(\Oo_T(d)) =nd+1$, then $h^1(\Ii_{Z,T}(d)) \ge1$. Hence $h^1(\Ii_{Z}(d)) \ge1$, since $T$ is arithmetically Cohen-Macaulay, and
 $S\in \TT_1(n,d;x)$. {Finally, since by assumption $\langle S\rangle=\PP^n$, we conclude that $S\in \TT(n,d;x)$.}
\end{proof}

\begin{proposition}\label{de2}
Assume $n\ge 2$ and $d\ge 5$ and set $$x = 1+ \left\lceil \frac{nd}{2}\right\rceil.$$ Fix a reducible rational normal curve $T=T_1\cup \cdots \cup T_s\subset \PP^n$, $s\ge 2$.  Assume the existence of $S\in \TT(n,d;x)'$ such that $S\subset T$.  Set $n_i:= \deg (T_i)$ and $x_i:= \#(S\cap T_i)$. Then:
\begin{enumerate}
\item[(i)] $S\subset T_{\reg}$;
\item[(ii)] $n$ is even and $d$ is odd;
\item[(iii)] every final component $T_i$ of $T$ has $n_i$ odd and $2x_i = n_id+1$.
\end{enumerate}
\end{proposition}

\begin{proof}
Set $W:= 2S\cap T$.
Note that $x_1+\cdots +x_s\ge x$ and that $x_1+\cdots +x_s= x$ if and only if $S\subset T_{\reg}$.
We have $n =n_1+\cdots +n_s$, $2x =nd+2$ if $nd$ is even and $2x =nd+3$ if $n$ and $d$ are odd. {Obviously $s\le d$ and hence $s-1<x$.}

\quad {\bf Step 1.}
We prove first of all that, for any $i$:
\begin{equation}\label{blabla}
2x_i\le n_id+1 
\end{equation}
Assume, by contradiction, that there exists $i$ such that $2x_i> n_id+1$ and set $S'=S\cap T_i$. 
Note that $h^1(\Ii_{2S'}(d)) =h^1(\Ii_{2S',{T_i}}(d)) $  
since $T_i$ is arithmetically Cohen-Macaulay.
Then since $h^0(\Oo_{T_i}(d)) =n_id+1$ and $\deg(2S')\ge n_id+2$, we have $h^1(\Ii_{2S'}(d))>0$ and hence $S\notin \TT(n,d;x)'$, a contradiction.

\smallskip

\quad {\bf Step 2.}
We prove now (i)  by contradiction.
 Set $S_1:= S\cap \mathrm{Sing}(T)$ and $S_2:= S\setminus S_1$. Since $T$ has at most $s-1$ singular points, $S_2\ne \emptyset$. 
We assume by contradiction that $S_1\neq\emptyset$.

For each $o\in \mathrm{Sing}(T)$ let $m(o)$ denote the number of irreducible components of $T$ passing through $o$. 
We saw that $T$ has Zariski tangent of dimension $m(o)$ and hence the connected component $W(o)$ of $W$ {supported at the point $o$} has degree $m(o)+1$. 
Thus, denoting $w= \deg (W)$, we have
\begin{equation}\label{bla}
w = 2\#(S_2) +\sum _{o\in S_1} (m(o)+1)\ge 2x+\#(S_1).
\end{equation}

If $\deg (W) \ge nd+4$, then fix $u\in S_2$ and set $S':= S\setminus \{u\}$. 
Note that $h^1(\Ii_{2S'}(d)) =h^1(\Ii_{2S',T}(d)) $  
since $T$ is arithmetically Cohen-Macaulay.
Then, {since $h^0(\Oo_T(d)) =nd+1$ and $\deg(2S')\ge w-2\ge nd+2$,} we have $h^1(\Ii_{2S'}(d))>0$ and hence $S\notin \TT(n,d;x)'$, a contradiction. 

Then we can assume \begin{equation}\label{assurdo}\deg (W)\le nd+3.\end{equation} 

\quad (a) Assume first $nd$ even.
Hence we have $2x=nd+2$. Then it follows that
$\#S_1 =1$, say $S_1=\{u\}$, and $T$ is nodal at $u$. Since $p_a(T)=0$, $T$ is connected, the irreducible components of $T$ are smooth
and $T$ is nodal at $u$, $T\setminus \{u\}$ has $2$ connected components. Call $T'$ and $T''$ the closures in $\PP^n$ of the two connected components of $T\setminus \{u\}$. Note that $\deg (W) =\deg (W\cap T') + \deg(W\cap T'')$ and $n =\dim \langle T'\rangle +\dim \langle T''\rangle$, either $\deg (W\cap T') \ge \dim \langle T'\rangle +2$ or  $\deg (W\cap T'') \ge \dim \langle T''\rangle +2$. Thus $S\notin \TT(n,d;x)'$ {and we have a contradiction.} We have proved (i) in this case.

\quad (b) Now assume $d$ odd and $n$ odd.
Then $2x =nd+3$, and by using \eqref{bla} and \eqref{assurdo} we get $S_1=\emptyset$. We have proved (i) in this case.

\smallskip

\quad {\bf Step 3.}
Since $d\ge 5$, a good ordering of the irreducible components of $T$
 and $s-1$ Mayer-Vietoris exact sequences give $h^1(\Ii_S(d-2)) =0$. Let $Z$ be a critical scheme for $S$, that is $h^1(\Ii_Z(d))>0$. Since $h^1(\Ii_T(1))=0$ and $h^2(\Ii_T)=h^2(\Oo_T(1)) =0$, the Castelnuovo-Mumford's lemma gives that $\Ii_T(2)$ is globally generated.
 Since $\Ii_T(2)$ is globally generated and every connected component of $Z$  has degree $\le 2$, $Q\cap Z =T\cap Z$ for a general $Q\in |\Ii_T(2)|$.
Since  $\Res_Q(Z)\subseteq S$ and $h^1(\Ii_S(d-2)) =0$ and $Q$ is arithmetically Cohen-Macaulay, the residual exact sequence with respect to $Q$ gives $h^1(\Ii_{Z\cap Q}(d)) =0$ and hence $Z\subset T$. Thus $Z\subseteq W$.  Since $T$ is arithmetically Cohen-Macaulay, we get $h^1(\Ii_{Z,T}(d)) >0$ and hence \begin{equation}\label{acca1}
h^1(\Ii_{W,T}(d)) >0.\end{equation}

\smallskip

\quad {\bf Step 4.}
We prove now (ii). Recall that, since $S\subset T_{\reg}$,
we have $x_1+\cdots +x_s= x$ and $n_1+\cdots +n_s=n$.

Assume by contradiction that $d$ is even. Recall the inequality \eqref{blabla} from Step 1.
If $d$ is even $2x_i\le n_id+1$, is equivalent to $2x_i\le n_id$, and this implies $2x\le nd$ which contradicts the assumption $2x =nd+2$. We have proved that $d$ is odd.

\smallskip

From now on, we assume $d$ odd. Recall \eqref{blabla}, and in particular:
$2x_i\le n_id+1$ for all odd $n_i$ and $2x_i\le n_id$ for all even $n_i$.

{Now assume $n$ odd by contradiction.}
 In particular, {since $2x =nd+3$, by \eqref{blabla}} we have $s\ge 3$ and there are at least three odd $n_i$ with $2x_i=n_id+1$. Let $T'$ be a minimal connected subcurve of $T$ such that $\deg (T'\cap W)\ge 2 + d\dim (\langle T'\rangle)$. Since $2x_i\le n_id+1$ for all $i$, {by \eqref{blabla},} and each subcurve $T''$ of $T$ has at least one final component (a final component of $T''$, not necessarily of $T$)
 the minimality of $T'$ gives $\deg (T'\cap W) = 2 + d\dim (\langle T'\rangle)$. 
 {It follows that $S\cap T'\in\TT(n,d;x)$ and, since $S\cap T'\subsetneq S$, we conclude that } $S\notin \TT(n,d;x)'$, a contradiction.
 
 Then we have proved (ii).
 
 \smallskip
 
 \quad {\bf Step 5.}
{We finally prove (iii). We know that $d$ is odd and $n$ is even by (ii).}

 Let $T_i$ any final component of $T$. Let $Y$ be the union of all other components of $T$. Since $T_i$ is a final component, $Y$ is connected . Then $\deg (Y) =\dim \langle Y\rangle$ and hence $Y$ is a, possibly reducible, rational normal curve in $\langle Y\rangle$, $\langle Y\rangle \cap \langle T_i\rangle$ is a point, $p$, and $\{p\}$ is the scheme-theoretic intersection of $T_i$ and $Y$. We proved {in Step 2} that $p\notin S$.
 Since $\langle S\rangle =\PP^n$ and $p\notin S$, {then} $\langle S\cap T_i\rangle = \langle T_i\rangle$ and $\langle S\cap Y\rangle =\langle Y\rangle$ and in particular $S\cap T_i\ne \emptyset$ and $S\cap Y\ne \emptyset$. Since $S$ is minimal and
 $T$ is arithmetically Cohen-Macaulay, $h^1(\Ii_{Z\cap T_i}(d)) = h^1(\Ii_{Z\cap Y,T}(d)) =0$. The  following Mayer-Vietoris type sequence on $T$
 \begin{equation}\label{eqmv1}
 0 \to \Ii_{W,T}(d)\to \Ii_{W\cap T_i,T_i}(d)\oplus \Ii_{W\cap Y,Y}(d)\to \Oo_p(d)\to 0
 \end{equation}
 is exact, because $p\notin S$. We proved that $h^1(\Ii_{Z\cap T_i,T_i}(d))=h^1(\Ii_{Z\cap Y,Y}(d))=0$.

{Assume by contradiction that $n_i$ is even.} Then we have $2x_i\le n_id$. The restriction map $H^0(\Ii_{W\cap T_i,T_i}(d))\to H^0(\Oo_p(d))$ is surjective, because $T_i\cong \PP^1$
and $\deg (W\cap T_i) \le \deg(\Oo_{T_i}(d))$. Thus \eqref{eqmv1} gives $h^1(\Ii_{W,T}(d)) =0$, a contradiction with \eqref{acca1}. Then we have proved that $n_i$ is even for
every final component $T_i$ of $T$. Hence we also have $2x_i = n_id+1$ and this conclude the proof.
\end{proof}

\section{Minimally Terracini finite sets in the plane} \label{sec-plane}

In this section we focus on the case of the plane.
We deduce from \cite{ep} the following result, which we will need in the sequel.

%\blu{OK va tutto bene per ogni grado controllato\\
%REMARK:
%P2 z fino a $2d+1$ lo schema difettivo Z di grado z deve avere retta
%Per d=2 uno schema di grado 3  non e' mai difettivo, di grado 4 solo se e' in una retta, di grado 5 solo se una retta lo interseca in grado almeno 4, in grado 6 solo se e'su conica (puo'essere frado 4,5,6 su una retta oppure su conica ma senza rette che danno $h^1$
%d=3, con z=8, o retta con grado almeno 5 o tutto contenuto in una conica
%con z=8 o retta che contiene almeno 5, o conica con intersezione completa,
%per z=9 compare oltre a $Z'\subset Z$ che da' $h^1$ il caso di Z in una cubica.
%idem per d=4 fino a $z \leq 13$}
%
%\blu{scriviamo il remark 6.0 oppure anche no e lasciamo cosi? DECIDERE (2)}
\begin{remark}\label{n2e1remark} 
Fix positive integers $d, z$ such that $z\leq 3d$.
Let $Z\subset \PP^2$ be a zero-dimensional {scheme}, $Z\ne \emptyset$. If $\deg(Z)=z$ and 
$d$ is the maximal integer
$t$ such that $h^1(I_{Z}(t)) >0$,
 then either there is line $L$ such that $\deg (L\cap Z)\ge d+2$
or there is a conic such that $\deg (Z\cap D)\ge 2d+2$ or $z=3d$ and $Z$ is the complete intersection of a plane cubic and a degree $d$ plane curve (see \cite[Remarque (i) p. 116]{ep}).
%\blu{bisogna forse mettere $d\ge4$? DOMANDA (2). $d\ge5$}
%%\rosso{*** Qui per $s=3$ chiede $z\ge s^3= 9$ e $d\ge s-3+z/s$ e quindi per (c) assumo $d\ge 5$;
%%per (b) assumo solo il default $d\ge 4$.}
\end{remark}

\begin{proposition}\label{n2a1} %\rosso{TO DO}
Fix integers $x>0$ and $d\ge 4$.

\quad (a) If $x\le d$, then $\TT(2,d;x)'=\emptyset$.

\quad (b) Let $S\in S(\PP^2,d+1)$. Then $S\in \TT(2,d,d+1)'$ if and only if $S$ is contained in a reduced conic $D$.  Moreover, if $D=R\cup L$ is reducible (with $L$ and $R$ lines), then $d$ is odd, $\#(S\cap R)=\#(S\cap L) = (d+1)/2$ and 
$S\cap R\cap L=\emptyset$.

\quad (c) Assume $d\ge 5$. Then $\TT(2,d;x)'=\emptyset$ for all $x$ such that $d+2 \le x < 3d/2$.
\end{proposition}
\begin{proof}
We prove (a) by contradiction. Assume $x\le d$ and consider $S\in \TT(2,d;x)'$. Let $Z$ be a critical scheme for $S$. We have $\deg (Z)\le 2x$ and $d$ is the maximal integer such that $h^1(\Ii_Z(d)) >0$ by Theorem \ref{rob1}.  
Then $\deg(Z)\le 2d$ and,
by Lemma \ref{obs1}, there is a line $L$ such that $\deg (Z\cap L)\ge d+2$. Thus  $h^1(\Ii_{Z\cap L}(d)) >0$. Since
$\langle S\rangle =\PP^2$, $S$ is not minimal, {a contradiction}.

The {\it if} implication of part (b)  follows from
Theorem \ref{a9.0} (iii). 

We prove now the other implication of (b).
Take $S\in \TT(2,d;d+1)'$ and let $Z$ be a critical scheme for $S$. By Lemma \ref{las1}, $Z_{\red}=S$. Assume that $S$ is not contained in a reduced conic. Since $\langle S\rangle =\PP^2$, $S$ is not contained in a double line, {therefore $S$ is not contained in a conic}. Hence Remark \ref{n2e1remark} implies that there is a line $L\subset \PP^2$ such that $\deg (L\cap Z)\ge d+2$ and hence $h^1(\Ii_{Z\cap L}(d)) >0$. Hence $S$ is not minimal. 
Finally  Proposition \ref{de2} gives the last part of (b).

We prove finally (c) {by contradiction}. Assume $d+2\le x<3d/2$ and let $S\in \TT(2,d;x)$ with $Z$ critical for $S$. Since $S$ is minimal $\#(S\cap L)\le (d+1)/2$ for all lines $L$ and $\#(S\cap D)\le 2d+1$ for each conic. Since $Z$ is critical, $\deg (Z\cap L)\le d+1$ for each line $L$ and $\deg (D\cap Z)\le 2d+1$ for any conic $D$. {Thus  since $\deg(Z)\le 3d-1$, by Remark \ref{n2e1remark} we have $h^1(\Ii_Z(d))=0$, a contradiction.}
\end{proof}

Just above the range covered by Proposition \ref{n2a1} we have the following examples.

\begin{example}\label{n2e1-exx} 
Assume $d=2k$, {for $k\in\NN$}, $d\ge 6$ and take $x:= 3 k$. Let $C\subset \PP^2$ be a smooth plane cubic and $T$ a smooth
plane curve of degree $k$. Take as $S$ the complete intersection $C\cap T$.
% Since $S$ is reduced, then %and a complete intersection of $C$ and another plane curve, $S\subset C_{\reg}$.
Set $Z:= C\cap 2T= 2S\cap C$. Since $\deg (Z) =3d$ and $h^0(\Oo_C(d)) =3d$,  
%and $C$ is arithmetically normal, 
then  $h^1(\Ii_{Z,C}(d)) =h^0(\Ii_{Z,C}(d))=1$. 
Since  $h^0(\Oo_C(d-3))) =3d-9\ge 3k=\#S$,
we get $h^1(\Ii_{S,C}(d-3)) =0$. Since $C$ is arithmetically normal, $h^1(\Ii_S(d-3)) =0$. Thus the residual exact sequence with respect to $C$ gives $h^1(\Ii_{2S}(d)) =h^1(\Ii_{{Z},C}(d)){=1}$. 
%Since $p_a(C)=1$, we get $h^1(\Ii_{(2S, C)}(d)) =1$. 
We also get $h^1(\Ii_{2S'\cap C, C}(d)) =0$ for all $S'\subsetneq S$, since $\deg(2S'\cap C)\leq 3d-2$.
Thus $S\in \TT(2,d;3d/2)'$. 
\end{example}

\begin{example}\label{n2e2} 
Take $d$ odd, $d\ge 7$, and set $x:= (3d+1)/2$. Let $C\subset \PP^2$ a smooth plane cubic. Take $S\subset C$ such that $\#S=(3d+1)/2$. By assumption $S$ is a Cartier divisor of $C$. Since $p_a(C)=1$ and $\deg (\Oo_C(d-3))) =3d-9>\#S$, then
$h^1(C,\Ii_{S,C}(d-3)) =0$. Since $C$ is arithmetically-normal, $h^1(\Ii_S(d-3)) =0$. Thus the residual exact sequence with respect to $C$ gives $h^1(\Ii_{2S}(d)) =h^1(\Ii_{2S\cap C,C}(d))$. Since $p_a(C)=1$, we get
$h^1(\Ii_{2S\cap C,C}(d)) =1$. We also have $h^1(\Ii_{2S'\cap C, C}(d)) =0$ for all $S'\subsetneq S$, since $\deg(2S'\cap C)\le 3d-1$, hence $S\in \TT(2,d;(3d+1)/2)'$.
\end{example}

\section{Minimally Terracini finite sets in $\PP^3$}\label{ultimasez}
Now we consider the case of finite sets of points in $\PP^3$.
The following proposition extends Remark \ref{n2e1remark} to the case of schemes of $\PP^3$.

%\blu{****inserisco nuova proof corretta****}
\begin{proposition}\label{sob1} 
Fix a positive integer $d$.  Let $Z\subset \PP^3$ be a zero-dimensional scheme such that $\langle Z\rangle =\PP^3$, its connected components have degree $\le 2$ and $z:= \deg(Z) \le 3d+1$. We have  $$h^1(\Ii_Z(d)) >0$$ if and only if one of the following cases occur:

\quad (i) there is a line $L\subset \PP^3$ such that $\deg (L\cap Z)\ge d+2$;

\quad (ii) there is a conic $D$ such that $\deg (D\cap Z)\ge 2d+2$;

\quad (iii) there is a plane cubic $T$ such that $\deg (T\cap Z)=3d$ and $T\cap Z$ is the complete intersection of $T$ and a degree $d$ plane curve.
\end{proposition}

\begin{proof}
{Set $S:= Z_{\red}$.}

Since the {\it if} part is trivial, we only need to prove the {\it only if} part.

We use  induction on $d$. The case $d=1$ is obvious,
since conditions $\deg (Z)\le 4$ and $\langle Z\rangle =\PP^3$ imply that $Z$ is linearly independent  {and hence $h^1(\Ii_Z(1))=0$.}

Assume $d\ge 2$ and that the proposition is true for lower degrees. 
If there is a plane $H$ such that $h^1(\Ii_{Z\cap H}(d)) >0$, then we may use Remark \ref{n2e1remark} {and we conclude}.

\smallskip

Now we 
assume that \begin{equation}\label{quasi-finito}h^1(\Ii_{Z\cap H}(d)) =0\text{ for any plane }H\subset \PP^3.
\end{equation}

%\blu{la proof non e' piu' per assurdo perche' troveremo il caso (i) OK!}
%and \rosso{we will get a contradiction, proving that this case is impossible.}

Take a plane $H\subset \PP^3$ such that  $w:= \deg (Z\cap H)$ is maximal. 
Since $\langle Z\rangle =\PP^3$
%$h^0(\Oo_{\PP^3}(1)) =4$, 
then we have $z\ge4$, and $w\ge 3$,
%per assunzione di massimalita' di H
and hence $\deg (\Res_H(Z)) =z-w\le 3(d-1)+1$. 
Since $h^1(\Ii_{Z\cap H}(d)) =0$ by \eqref{quasi-finito}, then the residual exact sequence with respect to $H$ gives $h^1(\Ii _{\Res_H(Z)}(d-1)) >0$. 
The inductive assumption applied to the scheme $\Res_H(Z)$ implies that we are {in one of the following cases:}

\quad case (i$^\prime$):
either there is a line $R$ such that
$\deg (R\cap \Res_H(Z)) \ge d+1$, 

\quad case (ii$^\prime$):
or there is a conic $E$ such that $\deg (E\cap \Res_H(Z))\ge 2d$, 

\quad case (iii$^\prime$):
or there is a plane cubic $C$ such that $\deg (C\cap \Res_H(Z))=3d-3$ and $\Res_H(Z) \cap C$ is the complete intersection of $C$ and a degree $d-1$ plane curve. 

\smallskip

We analyse separately the three cases in the following three steps (a), (b), (c).

\smallskip

 \quad {\bf Step (a).}
Assume first that we are in case (iii$^\prime$). %and we want to find a contradiction. 
Since
 $\deg (\Res_H(Z) \cap C)=3d-3$, then
 $z-w=\deg (\Res_H(Z))\ge 3d-3$.
On the other hand, since $\Res_H(Z) \cap C$ is contained in a plane, we also have
 $w\ge 3d-3$ and hence $z\ge 6d-6$. 
Now since $z\le 3d+1$, we get $d=2$.

Since $d=2$, we have $z\le 7$. Moreover 
since $h^1(\Ii_{\Res_H(Z)}(1)) >0$, then the scheme $\Res_H(Z)$ is linearly dependent and so we have
$w\ge \deg(\Res_H(Z))$, by the maximality assumption on $w$.
So we have $z-w\le3=2(d-1)+1$, and
by Lemma \ref{obs1},  it follows that there is a line $J$ such that  $\deg(J\cap \Res_H(Z))=(d-1)+2=3$.

{Take now a plane $M\supset J$ such that $w':= \deg(M\cap Z)$ 
is maximal. Since $\dim |\Ii_J(1)| =1$, we have $w'\ge 4$. 
%%\blu{perche' M gira e becca qualcosa fuori dai 3}
 We get $w=w'=4$ and $z=7$. Taking $M$ instead of $H$ and repeating the argument above, we have 
 $h^1(\Ii _{\Res_M(Z)}(1)) >0$, and again by Lemma \ref{obs1}, it follows that
there exists a line $K$ such that $\deg(K\cap \Res_M(Z))=3$, hence $\Res_M(Z) \subset K$.}

 If $\deg (K\cap Z)=4$ or $\deg (J\cap Z)=4$, then we are in case (i) and the theorem is proved.

Now we exclude the remaining case which is
\begin{equation}\label{ipotesis}\deg(Z\cap K)=\deg(Z\cap J) =3.
\end{equation}
Assume by contradiction \eqref{ipotesis}
 and 
 consider separately the following three possibilities:
either $J\cap K\ne \emptyset$ and $J\ne K$, or $K\cap J=\emptyset$, or $J=K$.

\smallskip

\quad (a1) Assume that $J\cap K\ne \emptyset$ and $J\ne K$. Recall that 
any connected component
 of $Z$ has degree $\le 2$, and clearly we have
$\deg (J\cap K)=1$. Hence the plane spanned by $J\cup K$ gives $w\ge \deg(J\cap Z)+\deg(J\cap K)-1=5$, a contradiction with $w=4$.

\smallskip

\quad (a2) Assume $K\cap J=\emptyset$. Since $\deg (Z)=7$ and $h^1(\Ii_Z(2)) >0$,  by Lemma \ref{a7} we have
 $\dim |\Ii_Z(2)|=h^0(\Ii_Z(2))-1\ge 3$. Take a general $Q\in |\Ii_Z(2)|$. The theorem of B\'ezout and the assumptions \eqref{ipotesis} imply that $J\cup K\subset Q$. Since $J\cap K=\emptyset$, $Q$ is not a an irreducible quadric cone or a double points. Moreover, since $Q$ is general, thent $Q$ is not the union of a plane containing $J$ and a plane containing $K$. Thus $Q$ is a smooth quadric. Since $J\cap K=\emptyset$, then $J$ and $K$ are contained in the same ruling of $Q$, say $J,K\in |\Oo_Q(1,0)|$. We have $h^1(Q,\Ii_{Z,Q}(2,2)) =h^1(\Ii_Z(2)) >0$.

%% %% \blu{se serve qui si puo fare il diagramma (come quello del theorem 3.1)*ma forse anche no*****che dici?**}
  
Note that, by using \eqref{ipotesis}, we have $h^1(K,\Ii_{Z\cap K,K}(2)) =h^1(J,\Ii_{Z\cap J,J}(2)) =0$. Since $\deg(Z)=7$, then the degree of $\Res_{J\cup K}(Z)$ is $1$, and hence it follows that
 $h^1(Q,\Ii_{\Res_{J\cup K}(Z),Q}(0,2)) =0$.
Now, taking cohomology
 of the residual exact sequence 
$$0\to \Ii_{\Res_{J\cup K}(Z),Q}(0,2)\to \Ii_{Z,Q}(2,2)\to \Ii_{(Z\cap J)\cup(Z\cap K),Q}(2,2)\to 0,$$
we obtain $h^1(Q,\Ii_{Z,Q}(2,2)) =0$, which is a contradiction.

\smallskip

\quad (a3) Assume finally that $J=K$. 
Recall that
 all the connected components of $Z$ have degree $\le 2$ and $S= Z_{\red}$. 
 From \eqref{ipotesis}
we deduce the following facts:
$\#(S\cap J)=3$, each connected component of $Z$ supported at $J$ has degree $2$ and none of them is contained in $J$. 
Moreover, since $\deg(Z)=7$, we have that $ S\setminus (S\cap J)$ is a simple point $p$. Let $H_1$ be a plane containing $J$ and not containing $p$. Set $Q_1:= 2H_1$ and consider
the residual exact sequence
with respect to $Q_1$ 
$$0\to \Ii_{\Res_{Q_1}(Z)}\to \Ii_{Z}(2)\to \Ii_{Z\cap Q_1,Q_1}(2)\to 0.$$
Since $J\subset \mathrm{Sing}(Q_1)$ and each connected component of $Z$ has degree $\le 2$, we have $Z_1:=Z\cap Q_1 =Z\setminus \{p\}$ and $\Res_{Q_1}(Z)=\{p\}$. Hence we have $h^1(\Ii_{\Res_{Q_1}(Z)})=h^1(\Ii_p)=0$.
It follows from the exact sequence that $h^1(\Ii_{Z_1,Q_1}(2))\ge h^1(\Ii_{Z}(2)) >0$ and hence $h^1(\Ii_{Z_1}(2)) >0$.

Fix now $p_1\in S\setminus \{p\}$ and let $A$ be the connected component of $Z_1$ supported at $p_1$. Take a plane $U$ containing $A\cup J$. Since $w=4$, by maximality we have $\deg (U\cap Z_1)\le \deg (U\cap Z) \le 4$, and hence $\deg (U\cap Z_1) = 4$. Since $\deg(\Res_U(Z_1)) =2$, then $h^1(\Ii_{\Res_U(Z_1)}(1)) =0$.
Thus taking the cohomology of the residual exact sequence with respect to $U$
$$0\to \Ii_{\Res_{U}(Z_1)}(1)\to \Ii_{Z_1}(2)\to \Ii_{Z_1\cap U,U}(2)\to 0.$$
we obtain $h^1(\Ii_{Z_1\cap U,U}(2))\ge 
h^1(\Ii_{Z_1}(2))>0$. This implies, by Lemma \ref{a9=}, that there is a plane $U$ such that $h^1(\Ii_{Z\cap U}(2))>0$, and this contradicts our assumption \eqref{quasi-finito}.

\medskip

\quad {\bf Step (b).}
Assume now that we are in case (ii$^\prime$).
Since there is a conic $E$ such that $\deg (E\cap \Res_H(Z))\ge 2d$, we get $w\ge 2d$ and $z-w\ge 2d$. It follows that $z\ge 4d$, which contradicts the assumptions $z\le3d+1$ and $d\ge2$.

\medskip

\quad {\bf Step (c).}
 Assume finally that we are in case (i$^\prime$), i.e.\
assume that there is a line $R$ such that
$\deg (R\cap \Res_H(Z)) \ge d+1$. If $\deg (R\cap Z)\ge d+2$, then we may take $L=R$ and we are in case (i) 
and the theorem is proved.

Now we assume that  $\deg(R\cap Z)=d+1$ and we will prove that either we are again in case (i), or we have a contradiction.

Since $\deg(R\cap Z)=d+1$, then we have $R\cap Z =R\cap \Res_H(Z)$.
%%%  \blu{NON E' vero vedi controesempio nell'email di edo(and $Z\cap R\cap H=\emptyset$). *****QUESTA PARENTESI TI TORNA? ****** }
 By the maximality assumption on $H$ we also know that $w=\deg(Z\cap H)\ge d+1$.

Take a general plane $M\supset R$ and consider the scheme $X:=Z\cap (H\cup M)$.  
Since $\deg(M\cap \Res_H(Z)) \ge \deg(R\cap \Res_H(Z))=d+1$, we have $\deg(X) \ge w+d+1\ge 2d+2$. Hence,
the hypothesis $\deg(Z)\le 3d+1$ implies that $\deg(\Res_{H\cup M}(Z)) \le d-1$. 
Then, by Lemma \ref{obs1}, we get
$h^1(\Ii_{\Res_{H\cup M}(Z)}(d-2))=0$. 

The residual exact sequence of $Z$ with respect to $H\cup M$: 
$$0\to \Ii_{\Res_{H\cup M}(Z)}(d-2)\to \Ii_{Z}(d)\to \Ii_{X,H\cup M}(d)\to 0$$
gives $h^1(\Ii_X(d)) = h^1(\Ii_{X,H\cup M}(d)) \ge h^1(\Ii_Z(d))
>0$.
%%%% \blu{where the first inequality is true because $X\cup M$ is arithmetically cohen. macualay }. 

Since $h^1(\Ii_{Z\cap M}(d)) =0$ by assumption \eqref{quasi-finito}, then we have also $h^1(\Ii_{X\cap M}(d))=0$, by Lemma \ref{a9=}.
The residual exact sequence of $X$ with respect to  $M$: 
$$0\to \Ii_{\Res_{M}(X)}(d-1)\to \Ii_{X}(d)\to \Ii_{X\cap M, M}(d)\to 0$$
gives $h^1(\Ii_{\Res_M(X)}(d-1)) >0$.  

{We consider now separately the two following cases: either the line $R$ is contained in $H$, or it is not contained.}

\smallskip

\quad (c1)  Assume $H\supset R$. Recall that $S= Z_{\red}$. Since each connected component of $Z$ has degree $\le 2$, we deduce the following facts: $\#(S\cap R)= d+1$, each connected component of $Z$ supported at a point of $S\cap R$ has degree $2$ and  no connected component of $Z$ is contained in $R$. 

Take general planes $H_1,H_2\in |\Ii_R(1)|$. Since $R=\mathrm{Sing}(H_1\cup H_2)$ and $H_1,H_2$ are general,  $Z'=Z\cap (H_1\cup H_2)$ is the union of the connected components of $Z$ which are supported at a point of $S\cap R$. 
Since $\deg(\Res_{H_1\cup H_2}(Z))\le 3d+1-2(d+1) =d-1$, by Lemma \ref{obs1} we have $h^1(\Ii_{\Res_{H_1\cup H_2}(Z)}(d-2)) =0$.
Then the residual exact sequence of $Z$ with respect to  $H_1\cup H_2$: 
$$0\to \Ii_{\Res_{H_1\cup H_2}(Z)}(d-2)\to \Ii_{Z}(d)\to \Ii_{Z', H_1\cup H_2}(d)\to 0.$$
gives $h^1(\Ii_{Z'}(d))=h^1(\Ii_{Z', H_1\cup H_2}(d))\ge h^1(\Ii_{Z}(d)) >0$.

Take a connected component $A$ of $Z'$. Since $\deg(A)=2$ and $\deg(A\cap R)=1$, there is a unique plane $H_3$ containing $A\cup R$.
Since $h^1(\Ii_{Z\cap H_3}(d)) =0$ by assumption \eqref{quasi-finito}, we have $h^1(\Ii_{Z'\cap H_3}(d))=0$ by  Lemma \ref{a9=}. Since $\deg(\Res_{H_3}(Z')) \le d$, we have $h^1(\Ii_{\Res_{H_3}(Z')}(d-1)) =0$ by Lemma \ref{obs1}. Thus the residual exact sequence of $Z'$ with respect to $H_3$:
$$0\to \Ii_{\Res_{H_3}(Z')}(d-1)\to \Ii_{Z'}(d)\to \Ii_{Z'\cap H_3, H_3}(d)\to 0$$
 gives $h^1(\Ii_{Z'}(d))=0$, a contradiction.  

\smallskip

\quad (c2) We assume now $H\nsupseteq R$. Thus $H$ contains at most one point of $S\cap R$.
For any $p\in R\cap S$ let $A_p$ denote the connected component of $Z$ supported at
$p$. 

Since $M$ is general and $S\cap R$ is  finite, $M\nsupseteq A_p$ for any $p\in S\cap R$. 
Recall that $X=Z\cap (H\cup M)$.
Thus if $S\cap H\cap R=\emptyset$, then we have
that $X =(Z\cap H)\cup (R\cap S)$ (as schemes), while if $R\cap H\cap S=\{p\}$, then $X$ is the union of $A_p$, the points $(S\setminus \{p\})\cap R)$ and the scheme $(Z\cap H)\setminus \{p\}$. 

Since $\deg(X) \ge 2d+2>d+1 =\deg (Z\cap R)$ there is a plane $U\supset R$ 
such that $ \deg(X\cap U)\ge d+2$. If $p\in S\cap R$ with $\deg(A_p)=2$ we  take as $U$ the plane spanned by $R\cup A_p$. 

We have $\deg(\Res_U(X))=\deg(X)-\deg(X\cap U) \le 3d+1-(d+2) =2(d-1)+1$. By Lemma \ref{obs1}
there is a line $J$ such that $\deg(\Res_U(X)\cap J)\ge d+1$. 

Since by construction we know that $\Res_U(Z)\cap R=\emptyset$, then $J\ne R$.

 If $\deg(J\cap Z)\ge d+2$, we take $L=J$ and we are in case (i) and the theorem is proved.
 Thus we may assume $\deg (J\cap Z)=d+1$ and we will find a contradiction. 
 
 If $J\cap R\ne \emptyset$, the plane $N$ spanned by $J\cup R$ proves that $w\ge \deg(N\cap X)\ge 2d+2$ and hence
$\deg(\Res_H(Z))=z-w \le d-1<\deg(Z\cap R)-1$, which is impossible since $\deg(Z\cap R\cap H)\le1$.

Now assume $J\cap R=\emptyset$. Fix a general $Q\in |\Ii_{J\cup R}(2)|$. Since any $2$ pairs of $2$ skew lines are projectively equivalent, $Q$ is smooth. Since $\Ii_{J\cup R}(2)$ is globally generated, $Q$ is general , each connected component of $Z$ has degree at most $2$ and $Z$ is finite, $Z\cap Q= Z\cap (J\cup R)$ (as schemes). Since $\deg(\Res_Q(Z)) \le 3d+1-2d-2=d-1$, we have by Lemma \ref{obs1} that $h^1(\Ii_{\Res_Q(Z)}(d-2)) =0$. 

Hence the residual exact sequence  with respect to  $Q$: 
$$0\to \Ii_{\Res_{Q}(Z)}(d-2)\to \Ii_{Z}(d)\to \Ii_{(Z\cap J)\cup (Z\cap R), Q}(d)\to 0$$
gives $h^1(\Ii_{(Z\cap J)\cup (Z\cap R)}(d))=h^1(\Ii_{(Z\cap J)\cup (Z\cap R), Q}(d))>0$. 

Taking a plane $N_1$ containing the line $J$ and exactly one point of $R\cap S$ we get 
$\deg(\Res_{N_1}(Z\cap J)\cup (Z\cap R))\le 2d+2-(d+2)=d$, hence by Lemma \ref{obs1} we have
$$h^1(\Ii_{\Res_{N_1}(Z\cap J)\cup (Z\cap R)}(d-1))=0,$$
on the other hand,
by assuption \eqref{quasi-finito} we know that
$h^1(\Ii_{N_1\cap Z}(d))=0$
and by Lemma \ref{a9=} we get
$h^1(\Ii_{N_1\cap ((Z\cap J)\cup (Z\cap R))}(d))=0$.

 Hence
 from
the following residual exact sequence:
$$0\to \Ii_{\Res_{N_1}(Z\cap J)\cup(Z\cap R)}(d-1)\to \Ii_{(Z\cap J)\cup(Z\cap R)}(d)\to \Ii_{N_1\cap((Z\cap J)\cup(Z\cap R) ), N_1}(d)\to 0$$
we obtain
 $h^1(\Ii_{(Z\cap J)\cup (Z\cap R)}(d))=0$, which is a contradiction. This ends the proof.
\end{proof}

Notice that if $z\le 3d$, case (iii) of the previous theorem never occurs since $\langle Z\rangle =\PP^3$.

Thanks to Proposition \ref{sob1},
we can easily prove  Theorem 
\ref{ooo1} which states the emptyness of the minimal Terracini loci $\TT(3,d;x)'$ for
$0<2x\le 3d+1$. 

\begin{proof}[Proof of Theorem \ref{ooo1}]    %%\blu{ACCORCIATO!!! controllare (2)}
%%Since the case $d=3$ is true by Lemma \ref{ooo3}, we may assume $d\ge 4$.
Consider $S\in \TT(3,d;x)'$ and let $Z$ be a critical scheme for $S$.
{By Lemma \ref{las1} we know that $Z_{\red}=S$ hence $\langle Z\rangle=\PP^3$.
Since $\deg(Z)\leq 2x\leq 3d+1$, we can apply Proposition \ref{sob1}.} 

{In any of the three cases there is a plane $H$ and a subset $S'=S\cap H$ which contradicts the minimality of $S$.}
\end{proof}

Now we will prove  Theorem \ref{n3.1}, which characterizes the elements of $\TT(3,d;1 +\lceil 3d/2\rceil)'$, i.e.\ {the sets of minimal cardinality which are} minimal Terracini with respect to $\Oo_{\PP^n}(d)$ in $\PP^3$.
Notice that one implication follows from
Theorem \ref{a9.0} (iii). 
By Proposition \ref{de2}, we also know that if $S$ is contained in a reducible rational normal curve, then $S\not\in \TT(3,d;1 +\lceil 3d/2\rceil)'$.

\begin{proof}[Proof of Theorem \ref{n3.1}]  
We only need to prove that any 
$S\in \TT(3,d;1 +\lceil 3d/2\rceil)'$  is contained in a rational
normal curve.

Given $d\ge 7$ and $x=1 +\lceil 3d/2\rceil$, we set
$\epsilon:=1$ if $d$ is even and $\epsilon:= 0$ if $d$ is odd. 
Given $S\in \TT(3,d;x)'$, let $Z$ be a critical scheme for $S$
and $z:= \deg (Z)$. Recall that $Z_{\red}=S$ and $z\le 3d+3-\epsilon$. 

Take a quadric $Q\in |\Oo_{\PP^3}(2)|$ such that $w:= \deg (Z\cap Q)$ is maximal. 

\smallskip

 \quad {\bf Step (a).}
In this step we want to prove that $Z\subset Q$.
Assume by contradiction that $Z\nsubseteq Q$. Since $h^0(\Oo_{\PP^3}(2)) =10$, $h^0(\Ii_A(2)) >0$ for every zero-dimensional scheme $A\subset \PP^3$ such that $\deg(A)\le 9$. Thus $w\ge 9$. By the minimality of $S$, we also have $h^1(\Ii_{Z\cap Q}(d))=0$,
hence $h^1(\Ii_{\Res_Q(Z)}(d-2))
>0$. 

Since $\deg(\Res_Q(Z))\le z-w\le 3(d-2)-\epsilon$, then Proposition \ref{sob1} implies that we are in one of the following cases:
\begin{itemize}
\item[(i)] there is a line $L$ such that $\deg (\Res_Q(Z)\cap L)\ge d$,
\item[(ii)]  there is a plane conic $D$ such that $\deg (\Res_Q(Z)\cap D)\ge 2d-2$;
\item[(iii)] $\epsilon =0$, $z=3d+3$, $w=9$ and $\Res_Q(Z)$ is the complete intersection of a plane cubic and a plane curve of degree $d-2$.
\end{itemize}

\quad (a1) First we exclude cases (ii) and (iii). Indeed, in both cases (ii) and (iii) there is a plane $U$ such that
$\deg (U\cap Z)\ge \deg(U\cap \Res_Q(Z)) \ge 2d-2$. Since $h^0(\Ii_U(2)) =4$, we have $w\ge \deg(U\cap Z)+3\ge 2d+1$ and hence we have
$\deg(\Res_Q(Z))=z-w\le 3d+3-(2d+1)<2d-2$, which is a contradiction.

\smallskip

\quad (a2) We assume now that we are in case (i), i.e.\ there is a line $L$ such that $$\deg(L\cap \Res_Q(Z))\ge d.$$

Note that, since $Z\not\subset Q$,
 there is a plane $H$ such that $L\subset H$ and $\deg (H\cap Z)\ge \deg(Z\cap R)+1\ge d+1$. We have $h^1(\Ii_{\Res_H(Z)}(d-1)) >0$, by the minimality of $S$, and $\deg (\Res_H(Z)) \le 3d+3-d-1=2d+2< 3(d-1)$. 
By applying Proposition \ref{sob1} to $\Res_H(Z)$, 
we are in one of the following cases:
\begin{itemize}
\item[(1)] there is a line $R$ such that $\deg (R\cap \Res_H(Z))\ge d+1$;
\item[(2)] there is a conic $D$ such that $\deg(D\cap \Res_H(Z)) \ge 2d$.
\end{itemize}

Now we consider separately these two possibilities (i1) and (i2).
\smallskip

\quad (a2.1) Assume we are in case (1), that is assume the existence of  a line $R$ such that $\deg (R\cap \Res_H(Z))\ge d+1$. The minimality of $S$ gives $\deg(R\cap Z)=d+1$ and $R\cap Z=R\cap \Res_H(Z)$.

Now we study the following cases: either $R=L$, or $R\ne L$ and $R\cap L\ne \emptyset$, or
 $R\cap L= \emptyset$.
 
\quad (a2.1.1) First assume $R=L\subset H$. 
Since $Z$ is critical, every connected component of $\Res_H(Z)$ supported at a point of $R$ is a simple point. Thus  we get
 $\#(S\cap R) \ge d+1$. Thus $h^1(\Ii_{2(S\cap R)}(d)) =h^1(\Ii_{2(S\cap R),R}(d)) >0$,   
 contradicting the minimality of $S$.

\quad (a2.1.2) Now assume $R\ne L$ and $R\cap L\ne \emptyset$. Consider the plane $M=\langle L\cup R\rangle $.  Since $\deg (L\cap R)=1$, then $\deg
(Z\cap M)\ge 2d$.  
Since $h^1(\Ii _{\Res_M(Z)}(d-1)) >0$ and $\deg (\Res_M(Z)) \le d+3$, there is a line $E$ such that $\deg
(E\cap \Res_M(Z)) \ge d+1$. As above we get $E\ne L$ and $E\ne R$. Take $Q'\in |\Ii_{E\cup L\cup R}(2)|$. Since $Z\nsubseteq
Q$ and $w$ is maximal,
{we have} $Z\nsubseteq Q'$. Hence $h^1(\Ii_{\Res_{Q'}(Z)}(d-2)) >0$ {and,} by Lemma \ref{obs1}, we have $\deg(\Res_{Q'}(Z))\ge d-1$.
Hence $z\ge (d-1)+\deg (Z\cap (L\cup R\cup E))=(d-1)+(2d+d+1-3)=4d-3$, a contradiction { since $d\ge7$.}

\quad (a2.1.3) Now assume $R\cap L=\emptyset$. Take  $Q''\in |\Ii_{R\cup L}(2)|$ such that $\deg (Z\cap Q'')$ is maximal. The maximality
of $w$ gives $Z\nsubseteq Q''$. Thus $h^1(\Ii_{\Res_{Q''}(Z)}(d-2)) >0$ and $\deg(\Res_{Q''})\le 3d+3-(d+1+d)=d+2\le 2(d-2)+1$. Hence there is a line $F$ such that $\deg (F\cap
\Res_{Q''}(Z))\ge d$.
{We conclude as in case (a2.1.2), using $L$, $R$ and $F$ instead of $L$, $R$ and $E$.}

\smallskip

\quad (a2.2) Assume that we are in case (2), that is there exists a conic $D$ such that $\deg(\Res_H(Z))\ge 2d$ and call $\langle D\rangle$ the plane spanned by $D$. Since $Z$ is minimally Terracini, $h^1(\Ii_{\langle D\rangle\cap Z}(d)) =0$ and hence $h^1(\Ii_{\Res_{\langle D\rangle}(Z)}(d-1)) >0$. Since $\deg(\Res_{\langle D\rangle}(Z)) \le d+3-\epsilon$,  Lemma \ref{obs1}
gives the existence of a line $R$ such that $\deg(R\cap \Res_{\langle D\rangle}(Z))\ge d+1$.  Thus we The minimality of $S$ implies $\deg(J\cap Z)=d+1$. The steps (a2.1.1), (a2.1.2)
and (a2.1.3) work verbatim taking $J$ instead of $R$.

\smallskip

 \quad {\bf Step (b).}
In step (a) we proved that $Z\subset Q$, hence we have $|\Ii_Z(2)|\ne \emptyset$.
In this step we prove that every quadric in $|\Ii_Z(2)|$ is integral.

{Assume by contradiction that $Z$ is contained in a quadric which is either not reduced, or reducible. We consider separately the two cases.}

\quad (b1) Assume first $Z\subset 2H$ where $H$ is a plane. Thus, since $S\subset Z$, we would have $S\subset H$, contradicting our definition of Terracini set.

\quad (b2) Assume now $Z\subset H\cup M$ where $H$ and $M$ are planes and $H\ne M$. With no loss of generality we may assume $\deg(Z\cap H)\ge \deg(Z\cap M)$.
The minimality of $S$ gives $h^1(\Ii_{\Res_H(Z)}(d-1)) >0$. Since $\deg(\Res_H(Z)) \le \lfloor z/2\rfloor <2(d-1)+1$, then Lemma \ref{obs1} implies that there is a line $L$ such that $\deg(L\cap \Res_H(Z)) \ge d+1$. 

Let $N$ be a general plane containing $L$. Since $\deg (\Res_{H\cup N}(Z)) \le z-d-1$, we have $h^1(\Ii_{\Res_{H\cup N}(Z)}(d-2)) =0$,
again by Lemma \ref{obs1}. The minimality of $S$ gives $Z\subset H\cup N$.
Taking different planes $N$ and $N'$ containing $L$, we get $S\subset (H\cup N)\cap (H\cup N')=H\cup L$. 
The minimality of $S$ implies $2\#(L\cap S)\le d+1$, i.e. $\#(S\cap L)\le \lfloor (d+1)/2\rfloor$. Since $\deg(\Res_H(Z)\cap L)) =d+1$, we get $d$ odd,  $H\cap Z\cap L=\emptyset$ and $\Res_H(Z) \subset L$. Since $\#(S\cap R)>1$ and $H\cap Z\cap L=\emptyset$, $L\nsubseteq H$.

Recall that $N$ is a general plane containing $L$. Again by the minimality of $S$ we have $h^1(\Ii_{\Res_N(Z)}(d-1)) >0$.
Since $\deg (\Res_N(Z)) \le 3d+3-d-1$, then Proposition \ref{sob1} implies that

\quad(I) either there is a line $R$ such that $\deg (R\cap \Res_N(Z))\ge d+1$, 

\quad(II) or there is a conic $D$ with $\deg (D\cap \Res_N(Z)) \ge 2d$.

{We analyse separately the two cases and we will show a contradiction in both cases.}

\quad (b2.1) Assume first the existence of {a conic} $D$ as in case (II). 

Since $|\Ii_D(2)|$ is globally generated and each connected component of $Z$ has degree at most $2$,
then $Q_1\cap Z=D\cap Z$ for a general $Q_1\in |\Ii_D(2)|$. Since $\deg(\Res_{N\cup Q_1}(Z)) \le z-(d+1)-2d\le 2$, we have $h^1(\Ii_{\Res_{N\cup Q_1}(Z)}(d-3)) =0$. The minimality of $S$ gives $Z\subset Q_1\cup N$. Since $N\cap Z\cap H=\emptyset$, and $Q_1\cap Z=D\cap Z$ and $Z\cap N=Z\cap L$, we get $Z\subset D\cup L$.

By the 
minimality of $S$ we have that $\#(S\cap D)\le d$ on the conic 
and $2\#(S\cap L)\le 
  d+2$ on the line, which implies
$\#(S\cap L)\le 
  \lfloor (d+1)/2\rfloor$.
  Then we would have
  $$1+\left\lfloor \frac{3d}2\right\rfloor=x\le d+ \left \lfloor \frac{d+1}2\right\rfloor$$
  which is false.

\smallskip

\quad (b2.2) Assume now the existence of a line $R$ as in case (I). 

Since $S$ is minimal, then $\deg(Z\cap R)=\deg(Z\cap L)=d+1$. Since $L\nsubseteq H$, and $S\subset H\cup L$,  we have $R\ne L$. Since $H\cap Z\cap S=\emptyset$,
we have $R\cap L\cap S=\emptyset$. Thus $\deg(Z\cap (R\cup L))=2d+2$. Since $\langle S\rangle =\PP^3$ and $Z$ is minimal, $R\cup L$ is not a conic, i.e. $R\cap L=\emptyset$.

Take a general $Q'\in |\Ii_{R\cup L}(2)|$.
Since $\Ii_{R\cup L}(2)$ is globally generated, then $Q'$ is smooth and $Z\cap Q' =Z\cap (R\cup L)$. Since $h^1(\Ii _{\Res_{Q'}}(d-2)) >0$ and $\deg (\Res_{Q'}(Z)) \le d+1$, Lemma \ref{obs1} implies that there is a line $E$ such that $\deg (E\cap Z)\ge d$.
Since $d$ is odd, we get $\#(E\cap S)=(d+1)/2$. Since $Z\subset H\cup L$ and $L\subset H$, then we get $L\cap E\ne \emptyset$. The conic $L\cup E$ contradicts the minimality of $S$.

\smallskip

 \quad {\bf Step (c).}
In steps (a) and (b) we proved that $Z$ in contained in a quadric $Q$ and that each quadric containing $Z$ is integral. 
Since $h^0(\Oo_{\PP^3}(2))=10$, for any degree $8$ scheme $W\subset Z$ we have $h^0(\Ii_W(2)) \ge 2$. Thus there is quadric $T\subset \PP^3$ such that $\deg (T\cap Z)\ge 8-\epsilon$ and $T\ne Q$.  In this step we prove that $Z\subset T$.

Assume {by contradiction} that  $Z\nsubseteq T$.
Since $\deg (\Res_T(Z)) \le 3(d-2) +1$,  the residual exact sequence of $T$ gives $h^1(\Ii_{\Res_T(Z)}(d-2)) >0$. 
First assume $\deg (\Res_T(Z)) =3d-5$ and that $\langle \Res_T(Z)\rangle$ is contained in a plane $M$. Since $Q$ is irreducible, $Q\cap M$
is a conic containing at least $\lceil (3d-5)/2\rceil$ points of $S$, contradicting the minimality of $S$. 

Since $\Res_T(Z)$ is not a scheme of degree $3d-5$ contained in a plane, then Proposition \ref{sob1}  implies that {we have the following cases:}

\quad($\alpha$) either there is a line $L_1$ such that $\deg (L_1\cap \Res_T(Z)) \ge d$, 

\quad($\beta$)
or there is a conic $D_1$ such that $\deg (D_1\cap \Res_T(Z)) \ge 2d-2$,

\quad($\gamma$)
or there is a plane cubic $C_1$ such that $\deg(C_1\cap \Res_T(Z)) \ge 3d-6$. 

\smallskip

{Now we analyse separatey the three cases an d we will get to a contradiction in any case.}

\quad (c1) Assume first the existence of the plane cubic $C_1$ as in case ($\gamma$). 

%Since $Q$ is integral, $\langle C_1\rangle\nsubseteq Q$ and hence $\deg(C_1\cap Q)=6<3d-6$, a contradiction.

Since $Z$ is contained in an integral quadric $Q$, then we have $\langle C_1\rangle\not\subset Q$.
Then $\deg(C_1\cap Q)\le6$ and this gives a contradiction because $6<3d-6$.

\smallskip

\quad (c2) Assume now the existence of the conic $D_1$ as in case ($\beta$).

The scheme $\Res_{\langle D_1\rangle}(Z)$
has degree $\le d+5-\epsilon$
and $h^1(\Ii_{\Res_{\langle D_1\rangle}(Z)}(d-1)) >0$, because $Z$ is critical. 
Thus by Lemma \ref{obs1},
there is a line $L_2$ such that $\deg(\Res_{\langle D_1\rangle}(Z)\cap L_2)\ge d+1$. Take a general plane $M\supset L_2$. We have $\deg(\Res_{M\cup \langle D_1\rangle}(Z)\le 4-\epsilon$. The minimality of $S$ gives $Z\subset M\cup \langle D_1\rangle$.
Then we proved that $Z$ is contained in a reducible quadric, which is impossible by
 step (b).

\smallskip

\quad (c3) Assume finally  the existence of the line $L_1$ as in case ($\alpha$). 

B\'ezout's theorem gives $L_1\subset Q$. Take a general plane $U\supset L_1$. Since each connected component of $Z$ has degree $\le 2$, then $L_1\cap Z=U\cap Z$. Since $\deg (\Res_U(Z)) \le 2d+3-\epsilon$
and $d\ge 6$, by Proposition \ref{sob1} it follows that: either there is a line $L_3$ such that $\deg (\Res_U(Z)\cap L_3)\ge d+1$, or there is a conic $D_3$ such that $\deg (D_3\cap \Res_U(Z)) \ge 2d$. 

We can again exclude the existence of $D_3$ %again by Proposition \ref{de2} (or 
following the same argument used  in step (c2). 

Now assume that there exists $L_3$ such that $\deg (\Res_U(Z)\cap L_3)\ge d+1$.
In this case we have $\#(S\cap (L_1\cup L_3)) \ge \lceil d/2\rceil +\lceil (d+1)/2\rceil =d+1$; we also get that $d$ is odd. Since $S$ is minimal, then $L_1\cap L_3=\emptyset$.
Thus the integral quadric $Q$ is not a cone, i.e. $Q$ is smooth.

{Then following the same argument used in step (a2.1.3) we get a contradiction (note that both steps (a2.1.2) and (a2.1.3) do not use the assumption $Z\nsubseteq Q$ made in step (a)). }

\smallskip

 \quad {\bf Step (d).}
By the previous steps, we know that $Z$ is contained in no reducible quadric and in infinitely many integral quadrics. Moreover, every quadric containing a degree $8-\epsilon$ subscheme of $Z$ contains $Z$. 

Let $Q$ be a general element of $|\Ii_Z(2)|$. 

Since in every pencil of quadrics at least one is singular, we can assume that $T$ is  a quadric cone containing $Z$. Since $Q$ is general, we may take $T$ such that $T\ne Q$. 
Call $o$ its vertex. Every line $L$ such that $\deg (L\cap Z)\ge 3$ is contained in $T$ and any union of $2$ lines of $T$ is a reducible conic, because they contain $o$.

Set $E:= Q\cap T$ as a scheme-theoretic intersection. Since $Z\subset T$ and $Z\subset Q$, then $Z\subset E$. Since $E$ is the complete intersection of $2$ quadric surfaces, the adjunction formula gives $\omega _E\cong \Oo_E$. The Koszul complex of the equations of $Q$ and $T$
gives $h^0(\Oo_E)=1$. 
{Hence by duality we have $h^1(\Oo_E)=1$. }

First assume $E$ integral, i.e.\ $E$ is an irreducible quartic curve. Since the rank $1$ torsion free sheaf $\Ii_{Z,E}(d)$ has degree $4d-\deg(Z)>0$, then $h^1(E,\Ii_{Z,E}(d)) =0$. Since $E$ is arithmetically Cohen-Macaulay, $h^1(\Ii_Z(d)) =0$, which is a contradiction. 

Then we may assume that $E$ is not integral. If $E$ is not reduced, it may have multiple components, but no embedded point. If $E_{\red} \ne E$, then $E_{\red}$ is a reduced curve of degree $\le 3$ containing $S$. Since $h^0(\Oo_E)=1$,
$E_{\red}$ is connected, hence Proposition \ref{de2} gives a contradiction. 

Thus the curve $E=E_{\red}$ is reduced and reducible. Each irreducible component of $E$ is either a line, or a smooth conic, or a rational normal curve. 

First assume $E=E_1\cup E_2$ with $E_1$ and $E_2$ reduced conics. Since $Z$ is critical and $S$ is minimal, then $h^1(\Ii_{\Res_{\langle E_i\rangle}(Z)}(d-1)) >0$ for $i=1,2$, and hence we have
$\deg (Z\cap E_1) +\deg (Z\cap E_2)-\deg (Z\cap E_1\cap E_2)\ge (2d+2)+(2d+2)-4=4d$, {contradicts the assumption $z\le 3d+3$, since $d\ge4$.}

Thus $E$ has at most one smooth conic among its irreducible components and it is not formed by $4$ lines through $o$. Hence there is a connected degree three curve $C\subset E$, which is either a rational normal curve, or a reducible rational normal curve.

{We consider now the following two cases: either $Z\nsubseteq C$, or $Z\subseteq C$.}

\smallskip

\quad (d1) First we assume that $Z\nsubseteq C$. 
Since $\Ii_C(2)$ is globally generated and every connected component of $Z$ has degree $\le 2$, for a general $Q'\in |\Ii_C(2)|$ we have
$Q'\cap Z =C\cap Z$. 
Hence it follows that
$h^1(\Ii_{\Res_{Q'}(Z)}(d-2)) >0$.
We write $E =L_4\cup C$ with $L_4$ a line.
We have $\Res_{Q'}(Z)\subset L_4$ and $\deg (\Res_{Q'}(Z))\ge d$. Take a general plane $M\supset L_4$. 

Since $h^1(\Ii_{\Res_M(Z)}(d-1)) >0$ {by minimality of $S$,}
and $\deg (\Res_M(Z)) \le 2d+3-\epsilon\le 3d$, then by Proposition \ref{sob1} we have that:

\quad(d1.1)
either there is a line $L_5\subset C$ such that $\deg (L_5\cap \Res_M(Z)) \ge d+1$, and this is impossible  because $E$ would be a union of $2$ reduced conics;

\quad(d1.2)
or there is a conic $D_4$ such that $\deg(\Res_M(Z)\cap D_4)\ge 2d$, and also in this case $E$ would be a union of $2$ reduced conics.

In both cases we find a contradiction and this complete the case $Z\nsubseteq C$.

\smallskip

\quad (d2) Now we assume $Z\subset C$. By Proposition \ref{de2}, we obtain that $C$ is a rational normal curve and this ends the proof of the theorem.
\end{proof}

We are going finally to prove our last main result, which is Theorem \ref{ceo1}.
We point out  that the bound {in Theorem \ref{ceo1}} 
is sharp, as shown in the following example, 
which implies that
 $\TT(3,d;2d)'\ne \emptyset$ for all $d\ge 5$.

\begin{example}\label{ex4d}
Take $d\ge 5$. Let $C\subset \PP^3$ be a smooth linearly normal elliptic curve.
%Since $C$ has genus $1$, $h^1(C,\Ll)=0$ for every line bundle on $C$ such that $\deg (\Ll)>0$. Moreover $\Ll$ is a spanned line bundle on $C$ if $\deg (\Ll)\ge 2$ and a very ample line bundle if $\deg (\Ll)\ge 3$.
%Recall also that $h^0(\Oo_C)=1$, $\omega _C\cong \Oo_C$ and $h^1(\Oo_C)=1$. \rosso{ok}
%\rosso{TO DO}
%Since $\Oo_C(d)$ has even degree and $C$ has genus $g=1$ there are $2^g =4$ line bundles $\Ll$ on $C$ such that
%$\Ll^{\otimes 2}\cong \Oo_C(d)$. Since $\deg (\Ll) =d\deg(C)/2 \ge 3$, $\Ll$ is very ample. 
{Let $\Ll$ be a line bundle on $C$ such that
$\Ll^{\otimes 2}\cong \Oo_C(d)$. Since $\deg(\Ll)=2d$ and $C$ has genus 1, $\Ll$ is very ample. 
}

Fix any $S\subset |\Ll|$ formed by $2d$ points. We will show that $S\in \TT(3,d;2d)'$.
Obviously $\langle S\rangle =\PP^3$.
Since $2S\cap C \in |\Oo_C(d)|$, we have 
$h^i(\Ii_{2S\cap C,C}(d))  =1$, $i=0,1$.

The curve $C$ is the  smooth complete intersection of $2$ quadric surfaces, say $C=Q\cap Q'$. Clearly $Q$ and $Q'$ are smooth at each point of $S$ and 
 $\Res_{Q}(2S)=S$ and $\Res_{Q'}(2S\cap Q)=S$, hence the residual exact sequence with respect to $Q$ in $\PP^3$ and of $C$ in $Q$ gives:
\begin{equation}\label{eqex41}
0\to \Ii_S(d-2)\to \Ii_{2S}(d)\to \Ii_{2S\cap Q,Q}(d)\to  0,
\end{equation}
\begin{equation}\label{eqex42}
0\to \Ii_{S,Q}(d-2)\to \Ii_{2S\cap Q,Q}(d)\to \Ii_{2S\cap C,C}(d)\to  0.
\end{equation}

Since $d\ge 5$, we have $\#S =2d< 4d-8=\deg (\Oo_C(d-2))$. 
Thus $h^1(\Ii_{S,C}(d-2)) =0$. 
Since $C$ is arithmetically Cohen-Macaulay, we have $h^1(\Ii_S(d-2)) =0$,
and hence $h^1(\Ii_{S,Q}(d-2))=0$. 
Using \eqref{eqex42} and \eqref{eqex41}, we get $h^1(\Ii_{2S}(d)) =1$ and $h^0(\Ii_{2S}(d)) \ge 1$. 

Take now $S'\subsetneq S$. 
Since $\deg (2S'\cap C) <4d$, we have $h^1(\Ii_{(2S'\cap C,C}(d)) =0$.
Moreover $h^1(Q,\Ii_{S',Q}(d-2))=0$,
by {Lemma} \ref{a9=}.
Hence, using again \eqref{eqex42} and  \eqref{eqex41} (with $S'$ instead of $S$), 
we get $h^1(\Ii_{2S'}(d)) =0$. 

Thus $S\in \TT(3,d;2d)'$.
\end{example}

%In the next remark we first recall some notation related to zero-dimensional schemes of the plane and useful for the sequel.}
%\begin{remark}\label{numerical}
%Given a zero-dimensional scheme $Z\subset \PP^2$, let $t$ be the minimal degree of a plane curve containing $Z$ and $\tau=\tau(Z)$  the maximal integer such that $h^1(\Ii_Z(\tau)) >0$.  Then there are integers
% $n_0\ge \cdots \ge n_{t-1} \ge t$ (defined in terms of the Hilbert function and called {\it numerical character}, see \cite{e,ep}) such that
%$n_0 =\tau +2$ and $\deg (Z) =\sum _{i=0}^{t-1} (n_i-i)$, i.e.
%$\sum _{i=0}^{t-1} n_i=z+ \binom{t}{2}.$
%%\begin{equation}\sum _{i=0}^{s-1} n_i=z+ \binom{s}{2}.\label{sum-n_i}\end{equation}
%\end{remark}
%

From the previous example we can deduce the following remark.
\begin{remark}\label{integral complete intersection}
Fix integers 
$x<2d$.  
Let $E\subset \PP^3$ be an integral complete intersection of two quadric surfaces.
Let $S$ be a collection of $x$ points on $E$, then $h^1(\Ii_{2S}(d))=0$.
\end{remark}

The following technical lemma generalizes Remark \ref{integral complete intersection}
to reducible quartic curves satisfying further suitable conditions.

\begin{lemma}\label{lemma-riducibile}
Fix  $d\ge5$.
Let $T \subset \PP^3$ be a reduced curve with $\deg(T)\le4$ and such that any irreducible component of $T$  is a line or a conic or a rational normal cubic. Assume also that no plane contains a subcurve of $T$ of degree $\ge 3$.
Let $S\subset T$ be a collection of points such that  $\#(S)\le 2d-1$ and
\begin{itemize}
\item 
$\#(S\cap L)\le
\lceil d/2\rceil$ for any line $L\subseteq T$
\item
$\#(S\cap C)\le d$ for any conic $C\subseteq T$
\item
$\#(S\cap D)\le (3d+1)/2$ for any  rational normal cubic $D\subseteq T$.
\end{itemize}
Let $Z\subset T$ be a zero-dimensional scheme such that $Z_{\red} =S$, any connected component of $Z$ has degree $\le2$, $Z$ is contained in an integral quadric surface and $Z$ is not contained in any reducible quadric.
Then $h^1(\Ii_{Z}(d))=0$.
\end{lemma}

\begin{proof}
Since $h^1(\Ii _T(t))=0$ for all $t\ge 5$, it is sufficient to prove that $h^1(\Ii_{Z,T}(d)) =0$. We already analized all cases with $\deg (T)\le 3$ and $T$ connected. Thus we may assume  that $T$ is connected and $\deg (T)=4$.

Consider a good ordering $T_1,\dots ,T_s$ of the irreducible components of $T$ and set $Y= T_1\cup \cdots \cup T_{s-1}$. The components $T_1$ and $T_s$ are final components and for every final component $T_i$ of $T$ there is a good ordering with $T_i$ as its first component. Thus, changing if necessary the good ordering, we may assume $\deg (T_1)\ge \deg(T_s)$. Thus $\deg (T_s) \le 2$
and $\deg (T_s) =2$ if and only if $s=2$ and $\deg (T_1)=2$. This case is excluded, because $T$ would be contained in a reducible quadric.

Hence $\deg(T_1)\ge \deg(T_s)=1$.
Set $E:=T_{s}\cap Y$ (scheme-theoretic intersection).  Since $T$ contains no plane subcurves of degree $\ge 3$, then %\blu{piu' preciso- esempio 4 rette} 
we can assume, up to choosing a good ordering that $\deg (T_s\cap Y)\le 2$. 
 Set $e:= \#(S\cap E)$ and $z:= \deg (Z)\le 2(\#S)$. Note that $\#S =\#(S\cap T_s) +\#(S\cap Y)-e$.
We have the  following Mayer-Vietoris type sequence on $T$
 \begin{equation}\label{MV}
 0 \to \Ii_{Z,T}(d)\to \Ii_{Z\cap T_s,T_s}(d)\oplus \Ii_{Z\cap Y,Y}(d)\to \Ii_{Z\cap E, E}(d)\to 0.
 \end{equation}

\quad (a) Assume $\#(S\cap T_s)\le \lceil d/2\rceil -1$. Thus $h^1(\Ii_{E\cup (Z\cap T_s),T_s}(d)) =0$,
since $\deg(E\cup (Z\cap T_s))\le 2+2(\lceil d/2\rceil -1)$.
Then
 the restriction map $H^0(\Ii_{Z\cap T_s,T_s}(d))\to H^0(\Ii_{Z\cap E, E}(d))$ is surjective.
Thus the exact sequence \eqref{MV} gives $h^1(\Ii_{Z,T}(d)) =0$ and we conclude.

\quad (b) Assume $\#(S\cap T_s) =\lceil d/2\rceil$. If $S\cap Y\cap T_s= \emptyset$, then we have $\Ii_{Z\cap E, E}(d)=\Oo_E(d)$ and we conclude as in step (a). 
Thus from now on we assume $S\cap T_s\cap Y\ne \emptyset$. Let
$M$ be a plane containing $L_s$ such that $\deg (Z\cap M)$ is  maximal.

\quad (b1) Assume  that $M$ contains another irreducible component, $T_i$, of $T$. Since $T$ contains no planar subcurve of degree $\ge 3$, $\deg (T_i)=1$ and $T_i$
is unique in $M$. Since $T_s\cup T_i$ is a conic, $\#(S\cap (T_s\cup T_i)) \le d$. The closure $A$ of $T\setminus (T_s\cup T_i)$ is either a reduced conic or the union of $2$ disjoint lines. The first case is excluded, because
$T$ is not contained in a reducible quadric. Now assume that $A$ is the union of $2$ disjoint lines, say $A=L\cup R$. The lines $L$ and $R$ are final components of $T$. By step (a) we may assume $\#(S\cap L) =\#(S\cap R) =\lceil d/2\rceil$.
Thus %\blu{questo ok se $T_i$ era l'unica che intersecava $T_s$. o c'e' un altro motivo?} 
$L\cap T_s =L\cap R=\emptyset$. Let $Q$ be the unique quadric containing $L\cup R\cup T_s$. Since $L\cap R=\emptyset$, $Q$ is a smooth quadric. Changing if necessary the names of the $2$ rulings of $Q$ we may assume
$L\cup R\cup T_s\in |\Oo_Q(3,0)|$. Since $T_i$ meets each connected component of $L\cup R\cup T_s$, B\'ezout's theorem gives $T_i\subset Q$ and $T_i\in |\Oo_Q(0,1)|$. Let $Z'\subset Q$ be the residual of $Z$ with respect to the divisor $L\cup R\cup T_s$.
It is sufficient to prove that $h^1(Q,\Ii _{Z'}(d-3,d)) =0$. Since $T_i\cup T_s$ is a reducible conic, $\#(S\cap T_i\cup T_s)\le d$ 
and hence $\#(S\cap T_i) \le d-\lceil d/2\rceil$ with strict inequality if $S\cap T_i\cap T_s\ne \emptyset$.
Thus $\deg (Z') \le 4d-2 -6\lceil d/2\rceil \le d-2$ and hence $h^1(\Ii_{Z',Q}(d-3,d)) =0$, {and we conclude that $h^1(\Ii_{Z,T}(d))=0$.}

\quad (b2) Assume that $T_s$ is the unique connected component of $T$ contained in $M$. Thus $\deg ((Y\cap (M\setminus T_s)) \le 3$. Hence $h^1(\Ii_{Z\cap M}(d)) =0$. By the residual exact sequence with respect to $M$ it is sufficient to prove that
$h^1(\Ii_{\Res_M(Z)}(d-1)) =0$. Assume by contradiction that $h^1(\Ii_{\Res_M(Z)}(d-1)) >0$. Since $\deg (M\cap Z)>\deg (Z\cap T_s)$, we have $\deg (\Res_M(Z)) \le 4d-2-2\lceil d/2\rceil -1 \le 3(d-1)$. Since $T$ contains no plane curve of degree $\ge 3$,  Proposition \ref{sob1} gives that either there is a line $L_1$ such that $\deg (L_1\cap \Res_M(Z)) \ge d+1$ or there is a conic $D_1$ such that $\deg (D_1\cap \Res_M(Z)) \ge 2d$. 

\quad (b2.1) Assume first the existence of the line $L_1$. Since $\#(S\cap L_1)\le \lceil d/2\rceil$, we get $d$ odd and $\deg (Z\cap L_1)= d+1$. Since $\#(S\cap J)\le d$ for all conics $J\subset T$ and $d$ is odd, $L_1\cap T_s=\emptyset$. Let $A_1$ denote the closure of $T\setminus (L_1\cup T_s)$. Either $A_1$ is a reduced conic or it is the union of $2$ disjoint lines. We have $\#(S\cap (T\setminus (T_s\cup L_1)))\le d-2$. There is an integral quadric $Q$ containing $T_s\cup L_1$ and at least one point of $S\cap (T\setminus T_s\cup R_1))$ for each component of $A_1$. Thus $h^1(\Ii_{\Res_Q(Z)}(d-2)) =0$. Thus it is sufficient to prove that $h^1(\Ii_{Z\cap Q,Q}(d)) =0$. Since $L_1\cap T_s=\emptyset$, $Q$ is a smooth quadric. We get $h^1(\Ii_{Z\cap Q,Q}(d))=0$,
unless $Q$ contains another irreducible component of $T$. First assume $A_1\subset Q$. Since $Q$ is a smooth quadric, we get (for a suitable choice of the $2$ rulings of $Q$) that either $T\in |\Oo_Q(4,0)|$ (excluded, because $T$ is reduced and connected)  or $T\in |\Oo_Q(3,1)|$ or $T\in |\Oo_Q(2,2)|$, which are also excluded.
Now assume that $Q$ only contains one component, $R$, of $A_1$. Write $A_1=R\cup R_2$ and $A_2:= L_1\cup T_s\cup R$. Either $A_2\in |\Oo_Q(3,0)|$ or $A_2\in |\Oo_Q(2,1)|$. In both cases we get $h^1(Q,\Ii_{Z\cap A_2,Q}(d,d)) =0$. To conclude
the proof we need to consider $R_2\cap Z\cap Q$. We have $\deg (R_2\cap Z\cap Q) \le 4$ %(anche molto meno) 
and hence $h^1(Q,\Ii_{Z\cap Q,Q}(d,d)) =0$.

\quad (b2.2) Assume the existence of the conic $D_1$. Since $\#(S\cap D_1)\le d$, we get $\#(S\cap D_1) =d$ and hence $\deg (Z\cap D_1) =2d$. By step (b1) we may assume that if $D_1$ is reducible, then none of its component contains $\lceil d/2\rceil$ points of $S$. We get $T=D_1\cup R\cup T_s$ with $R$ a line and $\#(S\cap (T\setminus D_1\cup T_s)) \le d-1-\lceil d/2\rceil$. If $R$ is a final component of $T$, then we use step (a) and that $\#(R\cap S)<\lceil d/2\rceil$.
Now assume that $R$ is not a final component of $T$. Assume for the moment $T_s\cap D_1\ne \emptyset$. Since $T$ contains no degree $3$ planar subcurve,  $D_1\cup T_s$ is a reducible rational normal curve and we may find a quadric $Q_1$ containing $D_1\cup T_s$, but not $R$. To conclude in this case we need $\deg (\Res_{Q_1}(Z)) \le d-1$. We have $\#(S\cap R) \le 2d-1 - d-\lceil d/2\rceil$, and we can conclude.
Now assume $D_1\cap T_s=\emptyset$. Since $T$ is connected, $R$ meet $T_s$ and $D_1$ at a different point. In this case $T$ is contained in the reducible quadric $\langle R\cup T_s\rangle \cup \langle D_1\rangle$, a contradiction. \end{proof}

{We give now the proof of Theorem \ref{ceo1}, which states that $\TT(3,d;x)'$ is empty if
  $1+\lceil {3} d/2\rceil < x < 2d$.}
\begin{proof}[Proof of Theorem \ref{ceo1}:]
Assume {by contradiction} the existence of $S\in \TT(3,d;x)'$ and fix a critical scheme $Z$ of $S$. Set $z:= \deg (Z)\le 4d-2$.

 Set $Z_0=Z$. For any $i>0$,
let $Q_{i}$ be a quadric surface such that $z_{i}:=\deg(Z_{i-1}\cap Q_i)$ is maximal and set $Z_i:= \Res_{Q_i}(Z_{i-1})$. 
The sequence $\{z_i\}_{i\ge 1}$ is weakly decreasing.  Let $e$ be the maximal $i$ such that $z_i\neq0$. Then $z=z_1+\cdots +z_e$ and $Z_e=\emptyset$. 
Since $h^0(\Oo_{\PP^3}(2)) =10$, $z_i\ge 9$ for all $i<e$, hence we have $e\le (4d+6)/9$, for
$z\le 4d-2$. 
By Lemma \ref{a43}, since $Z$ is critical and  $S\in \TT(3,d;x)'$,  we have 
$h^1(\Ii
_{Z_{e-1}}(d-2e+2)) >0$. %% In particular $z_e\ge d-2e+4$.
%
%%\blu{cambio l'ordine degli step (e1) e (e2)}
%\smallskip

\quad  (I) Assume first $e\ge 2$, i.e.\ $Z$ is not contained in any quadric surface.
Since $h^1(\Ii_{Z_{e-1}}(d-2e+2))
>0$, then Proposition \ref{sob1} implies that either $z_{e}\ge 3(d-2e+2)+1$ 
%\blu{ci sarebbe il $+1$ ma forse non serve (0)} 
or there is a line
$L$ such that
$\deg (Z_{e-1}\cap L)\ge d-2e+4$ or there is a plane conic $D$ such that $\deg (Z_{e-1}\cap D)\ge 2d-4e+6$.

%First assume $z_e\ge 3(d-2e+2)+1$. 
%Since the sequence $z_i$ is non-increasing, we get $z\ge e(3d-6e+6)$. 
%
%\blu{***}Set $f(t):=
%t(3d-\rosso{6t}+6)$. The real function $f(t)$ is increasing for $t\le (d+2)/{6}$, decreasing for $t>(d+2)/\rosso{6}$ and symmetric with respect
%to $(d+2)/\rosso{6}$. Since $h(2)=6d>z$, we get a contradiction.
%\blu{PROBLEMA (1)}
%
%
%\rosso{**ok ho verificato che funziona da d=13 in poi facendo i conti col computer**}
%\blu{**quindi va bene ma dire meglio ad esempio cosi:}

\quad  (I.a)
First assume $z_e\ge 3(d-2e+2)+1$.  Since the sequence $z_i$ is weakly decreasing, we get $z\ge e(3d-6e+6)$. 
It is easy to check that $e(3d-6e+6)> 4d-2$ for any $d\ge 13$ and $2\le e \le (4d+6)/9$. This contradicts our hypotesis.

%\blu{il 13 viene da quando questa frazione diventa piu' piccola di $d/2$ i conti non cambiano se si mette o no il $+1$}

\quad  (I.b)
Now assume the existence of a plane conic $D$ such that $\deg (Z_{e-1}\cap D)\ge 2d-4e+6$. 
Since $h^0(\Ii _D(2)) =5$, we get
$z_i\ge (2d-4e+6)+4$ for all $i<e$. 
Thus 
$z\ge e(2d-4e+10)-4$.
It is easy to check that 
$e(2d-4e+10)-4> 4d-2$ for any $2\le e \le (4d+6)/9$, and this gives again a contradiction. 

%\rosso{
%$e(2d-4e+6)> 4d-2$ for any $d\ge 22$ and $3\le e \le (4d+6)/9$. quindi resterebbe solo il caso $e=2$ da fare qui
%}

\quad  (I.c)
Finally assume the existence of a line
$L$ such that
$\deg (Z_{e-1}\cap L)\ge d-2e+4$. Since $h^0(\Ii_L(2)) = 7$, we have $z_i\ge (d-2e+4)+6$  for all $i<e$. 
%le quadriche che contengono la retta sono una famiglia di 7, posso imporre 6 condizioni 
Hence $z\le
e(d-2e+10)-6$. 
%\blu{anche qui farei come prima}
%
%\rosso{no} Set $u(t):= t(d-2t+10)-6$. The function $u(t)$ is increasing if $t\le (d+10)/4$ and decreasing for $t>
%(d+10)/4$.
%Since $u(d/2) =5d-6$ and $u(4) =4d+8$, we get $e\in \{2,3\}$. 
%
%\rosso{
It is easy to check that 
$e(2d-4e+11)-6> 4d-2$ for any $4\le e \le (4d+6)/9$. Hence we get $e\in \{2,3\}$.

Let $H$ be a general plane containing $L$. Since each
connected component of $Z$ has degree $\le 2$, we may assume $Z\cap L =Z\cap H$.

\quad  (I.c1)
First assume
$e=3$. Since
$z_1\ge z_2\ge z_3\ge d-2$ and $z_1+z_2\ge \lceil 2z/3\rceil$, we have $\deg (\Res_{Q_1\cup Q_2\cup H}(Z)) \le z-\lceil 2z/3\rceil-(d-2)= \lfloor
z/3\rfloor -d+2 < d-3=(d-5)+2$, since $d\ge7$.
 Since $S$ is minimally Terracini, we get $Z\subset Q_1\cup Q_2\cup L$. Since $e>2$ and $H$ is
contained in a quadric surface, $Z\nsubseteq Q_1\cup H$. 
%\rosso{OK!}
Since $S$ is minimally Terracini, $h^1(\Ii _{\Res_{Q_1\cup H}(Z)}(d-3))
>0$. 
%
%
%Since $\deg (\Res_{Q_1\cup H}(Z)) =\le z_2\le 2(d-3)+1$,
%\blu{PROBLEMA (1) non capisco lo $z_2$, ma  viene lo stesso cosi': }
%
We have:
$\deg (\Res_{Q_1\cup H}(Z)) %=\deg (\Res_{H}(Z_1)) 
\le (z-z_1)-(d-2)\le$ 
$z-\lceil z/3\rceil-d+2\le
\frac{5d+2}{3}\le %2d-5=
2(d-3)+1$, for $d\ge17$.
%%\blu{per ora e' la stima piu' alta}
Hence there is a line $R$ such that $\deg (R\cap \Res_{Q_1\cup H}(Z))\ge
d-1$.
Taking a general plane containing $R$ and taking again the residual, we get $Z\subset Q_1\cup L\cup R$. 
But since $h^0(\Ii_{R\cup L}(2)) >0$ and $e\le 2$, we have a
contradiction.

\quad  (I.c2)
Now assume $e=2$ and hence $z_1\ge\lceil z/2\rceil$. 
%
%
%\rosso{POSSIAMO TAGLIARE DA QUI ***}
%
%We have $\deg (\Res_{Q_1\cup H}(Z)) \le \lfloor z/2\rfloor -d \le d-1\rosso{\le 2(d-3)+1}$.
%
%Since $S$ is minimally Terracini, we get \rosso{either} $Z\subset Q_1\cup L$, \rosso{or} $z=4d-2$, $z_1=2d-1$ and there is a line $L'$ with
%$\deg (L'\cap \Res_{Q_1\cup \rosso{H}}(Z)) =d-1$. 
%
%\rosso{In the latter case, taking} the union of $Q_1$ and a general quadric
%$Q$ containing
%$L\cup L'$ 
%we get $Z\subset Q_1\cup L\cup L'$. Since $h^1(\Ii_{\Res_Q(Z)}(d-2)) >0$ \rosso{ (by minimality of $S$) }and $\deg (Z\cap Q_1) =2d-1$, we get
%\blu{qui non capisco ??(2). perche' la conica?}
%$Z\subset D''\cup L'\cup L$ for some conic $D''$ and conclude \rosso{as in Claim 2.}
%\blu{in the proof of Lemma \ref{lemma-claim 2}
%}
%Now assume that
%\rosso{$L'$ does not exists, hence $Z\subset Q_1\cup L$.}
%\blu{ma dove si usa questo? DOMANDA (3) forse si puo' saltare il caso precedente?} 
%
%\rosso{FINO A QUI ***}
%
We have
$\deg (\Res_H(Z))
\le z-d$ and
$h^1(\Ii _{\Res_H(Z)}(d-1)) >0$. 

First assume $\langle \Res_H(Z)\rangle =\PP^3$. Since $z-d \le 3(d-1)+1$, Proposition
\ref{sob1} implies that either there is a plane cubic $T_3$ with $T_3\cap \Res_H(Z)$ the complete intersection of $T_3$ and a
degree
$d-1$ plane curve or there is a conic $T_2$ such that $\deg (T_2\cap \Res_H(Z))\ge 2d$ or there is a line $T_1$ such that $\deg
(T_1\cap\Res_H(Z))\ge d+1$. 

First assume the existence of $T_3$. Since $\deg (\Res_{H\cup \langle T_3\rangle}(Z))\le 1$, {by minimality of $S$} we
 get $Z\subset H\cup \langle T_3\rangle$, contradicting the assumption $e>1$. 
 
 Assume the existence of $T_2$.
Since $\deg (\Res_{H\cup \langle T_2\rangle}(Z))\le z-3d\le d-1$, we get $Z\subset H\cup \langle T_2\rangle$, again a contradiction.

Now assume the existence of $T_1$. Take a general quadric $U\in|\Ii_{L\cup {T_1}}(2)|$. Since $\deg (\Res_U(Z))\le z-2d-1\le 2(d-2)+1$, by Lemma \ref{obs1}
there is a line $R_1$ such that $\deg (R_1\cap \Res_U(Z))\ge d$. Take a general $U'\in |\Ii_{L\cup T_1\cup R_1}(2)|$. Since
$\deg (\Res_{U'}(Z)) \le z-3d-1<d$ and $S$ is minimally Terracini, $Z\subset U'$, contradicting the assumption $e>1$.

Now assume $\dim \langle \Res_H(Z)\rangle \le 2$. %%\blu{forse qui si usa che siamo nel secondo caso?} 
%%\rosso{da controllare ancora questa parte}
 The only new case is if $\deg (\Res_H(Z)) =3d-2$ and $\Res_H(Z)$ is contained
in a plane cubic $C$. Since $\deg (\Res_{\langle C\rangle}(Z)) \le d$, $S$ is not minimally Terracini.

\quad (II). Assume now $e=1$, that is $Z$ is contained in a quadric $Q$.

If $Q$ is reducible 
we argue as in step (b) of the proof of Theorem \ref{n3.1} and we get a contradiction.
So we can assume  that {$Z$ is not contained in any reducible quadric. In particular $Q$ is irreducible and reduced.}

Set $W_0:= Z$. {Take $D_1\in |\Oo_Q(2)|$, such that
$w_1=\deg (W_0\cap D_1)$
is maximal and set $W_1:= \Res_{D_1}(W_0)$.
For $i\ge 2$, we iterate the construction: choose
divisors $D_i\in |\Oo_Q(2)|$ such that 
$w_{i}:=\deg(W_{i-1}\cap D_i)$ is maximal
 and set $W_i:= \Res_{D_i}(W_{i-1})$. }
The sequence $\{w_i\}_{i\ge 1}$, is weakly decreasing.  
Let $c\ge1$ be the maximal $i$ such that $w_i\neq0$, i.e.\  $W_c=\emptyset$ and and $z= w_1+\ldots +w_c$.

By Lemma \ref{a43}, since $Z$ is critical for  $S$ minimal,  we have 
 $h^1(\Ii_{W_{c-1}}(d-2c+2)) >0$. % In particular $w_c=\deg(W_{c-1})\ge d-2c+4$. }\blu{***}
Since $\dim |\Oo_Q(2)|=8$, if $w_i\le 7$, then $w_{i+1} =0$ and $W_{i+1}=\emptyset$. Thus $w_i\ge 8$ for $1\le i<c$, hence
we get $c \le \frac{4d+5}{8}$, since $z\le 4d-2$.
%\rosso{CONTO PRECISO: $4d-2\ge z\ge 8(c-1)+1$}

%\blu{PROBLEMA (2): mi sembra ancora che puo' essere $c= d/2$}

%%\blu{faccio i casetti}

\quad  {(II.a)}
If $c=1$, then we have $Z\subset D_1=Q\cap Q'$ where $Q'$ is an integral quadric. Hence $D_1$ is a complete intersection of two quadrics.
If $D_1$ is integral, then  by Remark \ref{integral complete intersection} we have
 $h^1(\Ii_Z(d)) =0$, a contradiction. 
If $D_1$ is reducible we have again a contradiction by Lemma \ref{lemma-riducibile} and by the minimality of $S$.

\quad  {(II.b)}
Now we assume $c=\lceil d/2\rceil $. %\blu{Bisogna fare il ceiling!!}
Hence either $d$ is even and $h^1(\Ii_{W_{c-1}}(2)) >0$, {or $d$ is odd and $h^1(\Ii_{W_{c-1}}(1)) >0$. }

%\blu{PROBLEMA: bisogna fare anche il caso dispari ma e' facile.}

%%\blu{SOLO controllare:}
{First assume $d$ odd and $c=\lceil d/2\rceil$. Then we have
$8(\lceil d/2\rceil -1) +\deg (W_{c-1})\le 4d-2$, then $\deg (W_{c-1}) \le 2$, which is a contradiction. 
}
Now assume $d$ even and $c=d/2$.
Since $8(d/2 -1) +\deg (W_{c-1})\le 4d-2$, then $\deg (W_{c-1}) \le 6$. 
Thus either there is a line $L$ such that
$\deg (W_{c-1}\cap L) \ge 4$ or $\deg (W_{c-1}) =6$ and $W_{c-1}$ is contained in a conic $D$. 

First assume the existence of the line $L$ such that $\deg ((W_{c-1})\cap L)\ge 4$. B\'ezout's theorem implies $L\subset Q$. Since $h^0(\Ii_{L,Q}(2)) =6$, the maximality of the integer $w_{c-1}$ implies
$w_{c-1}\ge w_c+5\ge 9$. Thus $4d-2\ge (d/2-1)9+4$, a contradiction, since $d\ge7$.

Now assume $\deg (W_{c-1}) =6$ and that $W_{c-1}$ is contained in a conic $D$.  If $D$ is reducible we may assume that no irreducible component $J$ of $D$ satisfied $\deg (J\cap W_{c-1}) \ge 4$.
With these assumptions B\'ezout's theorem implies $D\subset Q$. Since $h^0(\Ii_{D,Q}(2)) =4$, the maximality of the integer $w_{c-1}$ gives $w_{c-1} \ge w_c+3=9$, which leads again to a contradiction.

\quad  {(II.c)}
{Now we may assume $2\le c<d/2$.}

Assume for the moment $w_c\ge 3(d-2c+2)$.  
Since the sequence $\{w_i\}$ is weakly decreasing, $4d-2\ge z\ge 3c(d-2c+2)$. Since $c<d/2$, we get $c=1$ a contradiction.

Now assume $w_c<3(d-2c+2)$. By applying Proposition \ref{sob1} we know that either there is a conic $D$ such that $\deg (D\cap W_{c-1}) \ge 2(d-2c+2) +2=2d-4c+6$, or there is a line $L$ such that $\deg (L\cap W_{c-1})\ge d-2c+4$. 

\quad (II.c1) In the first case, since $h^0(\Ii_{D,Q}(2)) =4$, %\rosso{ok perche' $D$ conica} 
we have $w_i\ge (2d-4c+6)+3$ for all $i<c$. 
Hence $z\ge c(2d-4c+9)-3$. 
Since $z\le 4d-2$, then we have again a contradiction.

\quad (II.c2) Assume now the existence of $L$. Since $h^0(\Ii_{L,Q}(2)) =6$, we get $w_i\ge (d-2c+4)+5$ for all $i<c$. Thus $z\ge c(d-2c+9)-5$. 
It is easy to check that
 $2\le c\le 3$, hence
 $\deg (L\cap Z)\ge d-2$. 
 
 Take a quadric $U\in |\Oo_Q(2))|$ containing $L$ and such that
 $\deg (Z\cap U)$ is maximal. 
 Since $h^0(\Ii_{L,Q}(2)) =6$, we have $\deg (L\cap U)\ge (d-2)+5=d+3$. Thus $\deg (\Res_U(Z)) \le 4d-2-d-3  =3(d-2) +1$. By Proposition \ref{sob1} either there is a plane cubic $E$ such that $\deg (E\cap \Res_U(Z))\ge 3(d-2)$
or there is a conic $F$ such that $\deg (\Res_U(Z)\cap F) \ge 2d-2$ or there is a line $R$ such that $\deg (\Res_U(Z)\cap R)\ge d$. In all cases (since $d\ge 5$) B\'ezout's theorem implies that $R$, $F$ and $E$ are contained in $Q$ (or at least all the components supporting
$Z$). Since $Q$ is an integral quadric, we exclude the plane cubic $E$. 

\quad (II.c2.1) Assume the existence of a conic $F$. Even if $Q$ is not assumed to be smooth, $F$ is a plane section of $Q$ %%and %hence $\deg(F\cap L)=1$ \blu{potrebbe anche intersecare nel punto di incontro di due rette, ma sarebbe comunque reducible rat curve, no? quindi in realta' potrebbe essere anche 2, ma va bene lo stesso.  vero?(5)}
and $F\cup L$ is a reducible rational normal curve. 

Thus $Z\nsubseteq F\cup L$. %\blu{Perche'? DOMANDA (6)} 

Since $\Ii_{F\cup L}(2)$
is globally generated, a general $Q'\in |\Ii_{F\cup L}(2)|$ has $Q'\cap Z =(F\cup L)\cap Z$ and hence $\Res_{Q'}(Z)\ne \emptyset$. Since $h^1(\Ii_{\Res_{Q'}(Z)}(d-2)) >0$ and $\deg (\Res_{Q'}(Z)) \le 4d-2-3d+4$ and $d\ge 7$, there is a line $R'$
such that $\deg (\Res_{Q'}(Z)\cap R') \ge d$. Since $\Ii_{F\cup L\cup R'}(t)$ is globally generated for, say, $t=4$, we get $Z\subset F\cup L\cup R'$.
{Hence we conclude by Lemma \ref{lemma-riducibile}.}

\quad (II.c2.2) Assume  finally the existence of the line $R$. Since each connected component of $Z$ has degree $\le 2$ and no line contains $d-2$ points of $S$, %\blu{questo perche'? d/2}, 
$R\ne L$.

\quad (II.c2.2.1) First assume $R\cap L\ne \emptyset$. Thus $H:= \langle R\cup L\rangle$ is a plane. Since $\deg (\Res_H(Z))\le 4d-2-2d+2$ and $h^1(\Ii_{\Res_H(Z)}(d-1)) >0$ either $\deg (\Res_H(Z))=2d$ and $\Res_H(Z)$ is contained in a conic $F_1$
or there is a line $R_1$ such that $\deg (R_1\cap \Res_H(Z))\ge 2$. In the first case we get $Z\subset L\cup R\cup F_1$ 
{and we conclude by Lemma \ref{lemma-riducibile}.}

\quad (II.c2.2.2) Now assume $R\cap L=\emptyset$. Take a general $Q_1\in |\Ii_{R\cup L}(2)|$. Thus $Q_1\cap Z =(R\cup L)\cap Z$. We get $h^1(\Ii_{\Res_{Q_1}(Z)}(d-2)) >0$ with $\deg (\Res_{Q_1}(Z)) \le 2d$. We get that either there is a conic $F_2$
with $\deg (F_2\cap \Res_{Q_1}(Z)) \ge 2d-2$ or a line $R_2$ such that $\deg (R_2\cap \Res_{Q_1}(Z))\ge d$. If $F_2$ exist, we get $Z\subset R\cup L\cup F_2$ {and we use Lemma \ref{lemma-riducibile}.}
If $R_2$ exists, we take a general $U_1\in |\Ii_{R\cup L\cup R_2}(3)|$ and get that $Z$ is contained in the union of $4$ lines. {Hence we conclude again by Lemma \ref{lemma-riducibile}.}
\end{proof}

\end{document}